\documentclass[a4paper,11pt]{article}
\usepackage{titlesec}
\usepackage{amsthm}
\usepackage{amssymb}
\usepackage{color}
\usepackage{mathtools}
\usepackage{geometry}

\usepackage{algorithm2e}

\usepackage{verbatim}
\numberwithin{equation}{section}

\newtheorem{theorem}{Theorem}[section]
\newtheorem{corollary}[theorem]{Corollary}

\newtheorem{lemma}[theorem]{Lemma}

\theoremstyle{definition}

\newtheorem{remark}[theorem]{Remark}
\usepackage{array}
\newcolumntype{C}[1]{>{\centering\arraybackslash}m{#1}}
\usepackage{graphics}
\usepackage{amssymb, latexsym, amsmath, amsfonts, epsfig}
\usepackage{bm}
\usepackage{graphicx}
\usepackage{color}
\usepackage{epstopdf}
\usepackage{comment}
\usepackage{soul}
\usepackage[normalem]{ulem}
\usepackage{float}
\usepackage{MnSymbol}
\newcommand{\K}{\mathcal{K}}
\newcommand{\F}{\mathcal{F}}
\newcommand{\dK}{\partial K}
\newcommand{\FK}{\F_{\dK}}
\newcommand{\ZF}{Z_{\F,0}}
\newcommand{\ZFd}{Z_{\F,0}^\dagger}

\usepackage{caption}
\date{\vspace{-6ex}}

\begin{document}

\newcommand{\Question}[1]{{\marginpar{\color{blue}\footnotesize #1}}}
\newcommand{\blue}[1]{{\color{blue}#1}}
\newcommand{\red}[1]{{\color{red} #1}}

\newif \ifNUM \NUMtrue

\title{Spectral approximation of elliptic operators by the Hybrid High-Order method}
\author{Victor Calo\thanks{Department of Applied Geology, Western Australian School of Mines, Curtin University, Kent Street, Bentley, Perth, WA 6102, Australia; Mineral Resources, Commonwealth Scientific and Industrial Research Organisation (CSIRO), Kensington, Perth, WA 6152, Australia. E-mail address: victor.calo@curtin.edu.au} 
\and
Matteo Cicuttin\thanks{University Paris-Est, CERMICS (ENPC), 77455 Marne la Vall\'ee cedex 2, and INRIA Paris, 75589 Paris, France. E-mail address: matteo.cicuttin@enpc.fr} 
\and 
Quanling Deng\thanks{Corresponding author. Curtin Institute for Computation and Department of Applied Geology, Western Australian School of Mines, Curtin University, Kent Street, Bentley, Perth, WA 6102, Australia. E-mail address: quanling.deng@curtin.edu.au} 
\and 
Alexandre Ern\thanks{University Paris-Est, CERMICS (ENPC), 77455 Marne la Vall\'ee cedex 2, and INRIA Paris, 75589 Paris, France. E-mail address: alexandre.ern@enpc.fr}} 

\maketitle

\begin{abstract}
We study the approximation of the spectrum of a second-order elliptic differential operator by the Hybrid High-Order (HHO) method. The HHO method is formulated using cell and face unknowns which are polynomials of some degree $k\geq0$. The key idea for the discrete eigenvalue problem is to introduce a discrete operator where the face unknowns have been eliminated.
Using the abstract theory of spectral approximation of compact operators in Hilbert spaces, we prove that the eigenvalues converge as $h^{2t}$ and the eigenfunctions as $h^{t}$ in the $H^1$-seminorm, where $h$ is the mesh-size, $t\in [s,k+1]$ depends on the smoothness of the eigenfunctions, and $s>\frac12$ results from the elliptic regularity theory. The convergence rates for smooth eigenfunctions are thus $h^{2k+2}$ for the eigenvalues and $h^{k+1}$ for the eigenfunctions. Our theoretical findings, which improve recent error estimates for Hybridizable Discontinuous Galerkin (HDG) methods, are verified on various numerical examples including smooth and non-smooth eigenfunctions. Moreover, we observe numerically in one dimension for smooth eigenfunctions that the eigenvalues superconverge as $h^{2k+4}$ for a specific value of the stabilization parameter. \textbf{Mathematics Subjects Classification}: 65N15, 65N30, 65N35, 35J05
\end{abstract}
%
%


\section{Introduction} \label{sec:intr} 

The Hybrid High-Order (HHO) method has been recently introduced
for diffusion problems in \cite{DiPEL2014arbitrary} and for 
linear elasticity problems in \cite{DiPEr2015hybridl}. 
The HHO method is formulated by introducing cell 
and face unknowns which are polynomials of some degree $k\geq0$ 
(some variations in the degree of the cell unknowns are possible; see~\cite{CoDPE2016bridging}).
The method is then devised from a local reconstruction operator and 
a (subtle) local stabilization operator 
in each mesh cell. 
This leads to a discretization method that 
supports general meshes (with polyhedral cells and non-matching
interfaces). Moreover, when approximating smooth solutions of second-order 
elliptic source problems, the method delivers 
error estimates of 
order $h^{k+1}$ in the $H^1$-seminorm and of order $h^{k+2}$ in the $L^2$-norm under full elliptic regularity. 
Positioning unknowns at the mesh faces is also a natural way to express locally in each mesh cell the balance properties satisfied by the model problem.
As shown in~\cite{CoDPE2016bridging}, the HHO method
can be fitted into the family of Hybridizable Discontinuous Galerkin (HDG) 
methods introduced in~\cite{CoGoL:09} (and thus to the Weak 
Galerkin method \cite{WangY:13}) and is also closely related to the nonconforming Virtual Element Method from \cite{AyLiM:16}.
The HHO method has undergone a vigorous development over the last few years;
we mention, among others, the application to advection-diffusion equations in \cite{DiPDE2015discontinuous}, to the Stokes equations in \cite{DPELS:16}, to the
Leray--Lions equations in \cite{di2016hybrid}, 
and to hyperelasticity with finite deformations in \cite{AbErP:17}.
The implementation of HHO methods is described in \cite{cicuttin2017implementation}. As already pointed out in \cite{DiPEL2014arbitrary,DiPEr2015hybridl}, 
the cell unknowns can be eliminated locally
in each mesh cell, leading to a global Schur complement problem 
with compact stencil in terms of the face unknowns. 

The goal of this work is to devise and analyze HHO methods for the discretization of the eigenvalue problem associated with a second-order elliptic differential operator. The key idea is to formulate the discrete eigenvalue problem by letting the mass bilinear form act only on the cell unknowns, whereas the stiffness bilinear form acts, as for the discrete source problem, on both cell and face unknowns. Thus, the first main contribution of this work is to identify the relevant HHO solution operator approximating the exact solution operator. We show that this can be achieved by introducing a purely cell-based operator, where the face unknowns have been eliminated by expressing them in terms of the cell unknowns. Note that the elimination process is reversed with respect to the usual approach for the source problem, where one ends up with a face-based discrete operator. While the present cell-based operator is not needed for actual computations, it plays a central role in the error analysis. Indeed, with this tool in hand, it becomes possible to analyze the approximation error on the eigenvalues and the eigenfunctions by means of the abstract theory of spectral approximation of compact operators in Hilbert spaces following the work of Va\v{\i}nikko \cite{vainikko1964asymptotic,vainikko1967speed}, Bramble and Osborn \cite{bramble1973rate,osborn1975spectral}, Descloux et al.~\cite{descloux1978spectral,descloux1978spectral2}, and Babu\v{s}ka and Osborn \cite{babuvska1991eigenvalue}. 
The second main contribution of this work is Theorem~\ref{thm:errest} and Corollary~\ref{cor:H1} which establish a convergence of order $h^{2t}$ for the eigenvalues and of order $h^{t}$ for the eigenfunctions in the $H^1$-seminorm, where $t\in [s,k+1]$ is the smoothness index related to the eigenfunctions and $s\in (\frac12,1]$ is the smoothness index resulting from the elliptic regularity theory. 
In the case of smooth eigenfunctions, we have $t=k+1$, leading to a convergence of order $h^{2k+2}$ for the eigenvalues and of order $h^{k+1}$  for the eigenfunctions in the $H^1$-seminorm.  These convergence orders are confirmed by our numerical experiments including both smooth and non-smooth eigenfunctions of the Laplace operator in one and two dimensions. We highlight that these convergence results are so far lacking for HDG methods (see the discussion in the next paragraph), so that the present work contributes to fill this gap.
Finally, the third contribution of this work is the numerical observation of a superconvergence of order $h^{2k+4}$ for the eigenvalues in one dimension whenever the stabilization parameter is chosen to be equal to $(2k+3)$. 

Let us put our results in perspective with the literature on the approximation of elliptic eigenvalue problems by other discretization methods. Following the early work in \cite{strang1973analysis}, it is well-known that using $H^1$-conforming finite elements of degree $k\ge1$ on simplicial meshes leads to convergence rates of order $h^{2k}$ for the eigenvalues and of order $h^k$ for the eigenfunctions (provided the eigenfunctions are smooth enough). We refer the reader to \cite{boffi2010finite} for a review on the finite element approximation of eigenvalue problems. Similar results were obtained more recently in \cite{antonietti2006discontinuous,giani2015hp} for discontinuous Galerkin (dG) methods. The analysis of the spectral approximation by mixed and mixed-hybrid methods was started in \cite{canuto1978eigenvalue,mercier1978eigenvalue,mercier1981eigenvalue} and expanded in \cite{DurGP:99,BofBL:00}. Hybridization techniques leading to an eigenproblem on the face unknowns were studied in \cite{cockburn2010hybridization} for Raviart--Thomas mixed finite elements; therein, it was also observed that the use of a local post-processing technique improves the accuracy of the computed eigenfunctions (see also \cite{Gardini:09} for the lowest-order case). The approximation of elliptic eigenvalue problems using the Virtual Element Method (VEM) was studied in \cite{gardini2017virtual}, where optimal convergence rates were obtained. The spectral approximation of elliptic operators by the HDG method was analyzed in~\cite{gopalakrishnan2015spectral}, leading to a convergence of order $h^{2k+1}$ for the eigenvalues; therein, a non-trivial post-processing using a Rayleigh quotient was also examined numerically leading to an improved convergence of order $h^{2k+2}$ for $k\ge1$. In contrast, the HHO approximation directly delivers a provable convergence of order $h^{2k+2}$ even for $k=0$.
Finally, let us mention the recent work in \cite{puzyrev2017dispersion,deng2017dispersion,calo2017dispersion} which studies numerically the optimally blended quadrature rules \cite{ainsworth2010optimally} for the isogeometric analysis \cite{hughes2014finite} of the Laplace eigenvalue problem and reports superconvergence of order $h^{2k+2}$ for the eigenvalue errors while maintaining optimal convergence of orders $h^k$ and $h^{k+1}$ for the eigenfunction errors in the $H^1$-seminorm and in the $L^2$-norm, respectively.

The rest of this paper is organized as follows. Section \ref{sec:ps} presents the second-order elliptic eigenvalue problem and briefly recalls the main abstract results we are going to use concerning the spectral approximation of compact operators in Hilbert spaces. Section~\ref{sec:hho} deals with the HHO discretization, first of the source problem and then of the eigenvalue problem. The algebraic realization of both problems is also presented. Section~\ref{sec:hho} additionally identifies the relevant notion of discrete solution operator for HHO methods and outlines the error analysis for the HHO discretization of the source problem. This analysis is based on the results of \cite{DiPEL2014arbitrary}, but we handle the case where the exact solution does not have full regularity. Section~\ref{sec:ea} is concerned with the error analysis for the HHO discretization of the eigenvalue problem and contains our main results.  
Section \ref{sec:num} presents our numerical examples. Finally, some concluding remarks are collected in Section \ref{sec:conclusion}.

\section{Functional setting} \label{sec:ps}
In this section, we present the second-order elliptic eigenvalue problem, and briefly recall the main abstract results on the approximation of the spectrum of compact operators in Hilbert spaces.
\subsection{Problem statement} 
\label{sec:statement}

We consider the following second-order elliptic eigenvalue problem: Find an eigenpair $(\lambda, u)$ with $\lambda\in\mathbb{R}_{>0}$ and $u:\Omega\rightarrow\mathbb{R}$ such that 
\begin{equation} \label{eq:pdee}
\begin{alignedat}{2}
-\Delta u & =  \lambda u &\quad &\text{in $\Omega$}, \\
u & = 0 &\quad &\text{on $\partial \Omega$},
\end{alignedat}
\end{equation}
where $\Omega \subset  \mathbb{R}^d$, $d\in\{1,2,3\}$, is a bounded open domain with Lipschitz boundary $\partial \Omega$ and $\Delta$ is the Laplacian. In weak form, the problem \eqref{eq:pdee} reads as follows: Find $(\lambda, u)\in \mathbb{R}_{>0}\times H^1_0(\Omega)$ such that 
\begin{equation} \label{eq:eigen_weak}
a(u, w) = \lambda b(u, w), \qquad \forall w \in H^1_0(\Omega),
\end{equation}
with the bilinear forms $a$ and $b$ defined on $H^1_0(\Omega)\times H^1_0(\Omega)$ and $L^2(\Omega)\times L^2(\Omega)$ as
\begin{equation}
a(v, w) = (\nabla v, \nabla w)_{L^2(\Omega)}, \qquad  b(v, w) = (v, w)_{L^2(\Omega)},
\end{equation}
where $(\cdot, \cdot)_{L^2(\Omega)}$ denotes the inner product in $L^2(\Omega)$ or in $L^2(\Omega;\mathbb{R}^d)$.
The eigenvalue problem \eqref{eq:pdee} has a countably infinite sequence of eigenvalues $(\lambda_j)_{j\ge1}$ (see, among many others, \cite[Sec.~9.8]{Brezis:11}) such that
\begin{equation}
0 < \lambda_1 < \lambda_2 \leq \lambda_3 \leq \cdots, \qquad \lambda_j\to +\infty,
\end{equation}
and an associated sequence of $L^2$-orthonormal eigenfunctions $(u_j)_{j\ge1}$ such that
\begin{equation}
(u_j, u_l)_{L^2(\Omega)} = \delta_{jl}, \qquad \forall j, l\ge1,
\end{equation}
with the Kronecker delta defined as $\delta_{jl} =1$ when $j=l$ and zero otherwise. 

The source problem associated with the eigenvalue problem~\eqref{eq:eigen_weak} is as follows: For all $\phi\in L^2(\Omega)$, find $u\in H^1_0(\Omega)$ such that 
\begin{equation} \label{eq:source_weak}
a(u,w) = b(\phi, w), \qquad \forall w \in H^1_0(\Omega).
\end{equation}
The solution operator associated with~\eqref{eq:source_weak} is denoted as
$T: L^2(\Omega) \rightarrow L^2(\Omega)$, so that we have $T(\phi)\in H^1_0(\Omega) \subset L^2(\Omega)$ and
\begin{equation} \label{eq:sourceT_weak}
a(T(\phi),w) = b(\phi, w), \qquad \forall w \in H^1_0(\Omega).
\end{equation}
By the Rellich--Kondrachov Theorem (see, e.g., \cite[Thm.~1.4.3.2]{Gr85}), 
$T$ is compact from $L^2(\Omega)$ to $L^2(\Omega)$. 
Moreover, the elliptic regularity theory
(see, e.g., \cite{Gr85,Savare_1998,Jochmann_1999})
implies that there is a real number $s\in (\frac12,1]$ so that 
$T \in \mathcal{L}(L^2(\Omega);H^{1+s}(\Omega))$.
The reason for introducing the solution operator $T$ is that 
$(\lambda, u)\in\mathbb{R}_{>0}\times H^1_0(\Omega)$ is an eigenpair for \eqref{eq:eigen_weak} if and only if $(\mu,u)\in\mathbb{R}_{>0}\times H^1_0(\Omega)$ with $\mu = \lambda^{-1}$ is an eigenpair of $T$. 

One can also consider the adjoint solution operator $T^*: L^2(\Omega) \rightarrow L^2(\Omega)$ such that, for all $\psi \in L^2(\Omega)$, $T^*(\psi)\in H^1_0(\Omega)$ and 
\begin{equation} \label{eq:sourceTadj_weak}
a(w,T^*(\psi)) = b(w,\psi), \qquad \forall w \in H^1_0(\Omega).
\end{equation}
The symmetry of the bilinear forms $a$ and $b$ implies that $T=T^*$; however, allowing more generality, we keep a distinct notation for the two operators. Since in general we have 
\begin{equation} \label{eq:tts}
(T(\phi),\psi)_{L^2(\Omega)} = a(T(\phi),T^*(\psi)) = (\phi,T^*(\psi))_{L^2(\Omega)},
\end{equation}
we infer that $T^*$ is the adjoint operator of $T$, once the duality product is identified with the inner product in $L^2(\Omega)$. Therefore, in the present symmetric context, the operator $T$ is selfadjoint.

\subsection{Spectral approximation theory for compact operators} 
\label{sec:theory}

Let us now briefly recall the main results we use concerning the spectral approximation of compact operators in Hilbert spaces. Let $L$ be a Hilbert space with inner product denoted by $(\cdot,\cdot)_L$, and let $T \in \mathcal{L}(L;L)$; assume that $T$ is compact. We do not assume for the abstract theory that $T$ is selfadjoint and we let $T^*\in\mathcal{L}(L;L)$ denote the adjoint operator of $T$. Let $T_n\in \mathcal{L}(L;L)$ be a member of a sequence of compact operators that converges to $T$ in operator norm, i.e., 
\begin{equation} \label{eq:normc}
\lim_{n \to +\infty} \|  T -  T_n \|_{\mathcal{L}(L; L)} = 0,
\end{equation}
and let $T_n^*\in \mathcal{L}(L;L)$ be the adjoint operator of $T_n$.
We want to study how well the eigenvalues and the eigenfunctions of $T_n$ approximate those of $T$. 
Let $\sigma(T)$ denote the spectrum of the operator $T$ and let $\mu \in \sigma(T)\setminus\{0\}$ be a nonzero eigenvalue of $T$. Let $\alpha$ be the ascent of $\mu$, i.e., the smallest integer $\alpha$ such that 
$\text{ker}(\mu I -  T)^\alpha = \text{ker}(\mu I -  T)^{\alpha+1}$,
where $I$ is the identity operator. Let also 
\begin{equation} \label{eq:def_Gmu}
G_\mu = \text{ker}(\mu I-T)^\alpha, \qquad G^*_\mu = \text{ker}(\mu I-T^*)^\alpha,
\end{equation} 
and $m = \text{dim}(G_\mu)$ (this integer is called the algebraic multiplicity of $\mu$; note that $m\ge \alpha$). 

\begin{theorem}[Convergence of the eigenvalues] \label{thm:eve0}
Let $\mu\in \sigma(T)\setminus\{0\}$. 
Let $\alpha$ be the ascent of $\mu$ and let $m$ be its algebraic multiplicity.
Then there are $m$ eigenvalues of $T_n$, denoted as $\mu_{n,1}, \cdots, \mu_{n,m}$, that converge to $\mu$ as $n \to +\infty$. Moreover, letting $\langle \mu_n \rangle = \frac{1}{m} \sum_{j=1}^m \mu_{n,j}$ denote their arithmetic mean, 
there is $C$, depending on $\mu$ but independent of $n$, such that 
\begin{equation} \label{eq:cv_eigenval}
\begin{aligned}
\max_{1 \le j \le m} |\mu - \mu_{n,j}|^\alpha & + |\mu - \langle \mu_n \rangle | \le C \Big( \sup_{\substack{0\ne \phi \in G_\mu\\ 0\ne \psi \in G_\mu^*}} \frac{|((T-T_n)\phi, \psi)_L|}{\| \phi \|_L  \| \psi \|_L } \\
& + \|(T-T_n)|_{G_\mu} \|_{\mathcal{L}(G_\mu;L)}  \|(T-T_n)^*|_{G_\mu^*} \|_{\mathcal{L}(G_\mu^*;L)} \Big).
\end{aligned}
\end{equation}
\end{theorem}

\begin{remark}[Convergence of the arithmetic mean]
Note that \eqref{eq:cv_eigenval} shows that for $\alpha\ge 2$, the arithmetic mean of the eigenvalues has a better convergence rate than each eigenvalue individually. 
\end{remark}

\begin{theorem}[Convergence of the eigenfunctions] \label{thm:efe0}
Let $\mu\in \sigma(T)\setminus\{0\}$ with  
ascent $\alpha$ and algebraic multiplicity $m$.
Let $\mu_{n, j}$ be an eigenvalue of $T_n$ that converges to $\mu$. Let $w_{n,j}$ be a unit vector in $\text{ker}(\mu_{n,j}I-T_n)^\ell$ for some positive integer $\ell \le \alpha$. Then, for any integer $r$ with $\ell \le r \le \alpha$, there is a vector $u_r \in \text{\em ker}(\mu I- T)^r \subset G_\mu$ such that
\begin{equation}
\|u_r - w_{n,j}\|_L \le C \|(T-T_n)|_{G_\mu}\|_{\mathcal{L}(G_\mu; L)}^{\frac{r - \ell + 1}{\alpha}},
\end{equation}
where $C$ depends on $\mu$ but is independent of $n$.
\end{theorem}

\section{HHO discretization} \label{sec:hho}

In this section we present the discrete setting underlying the HHO discretization and then we describe the discretization of the source problem~\eqref{eq:source_weak} and of the eigenvalue problem~\eqref{eq:eigen_weak} by the HHO method. The HHO discretization of the source problem has been introduced and analyzed in \cite{DiPEL2014arbitrary}; herein, we complete the error analysis by addressing the case where the solution has minimal elliptic regularity pickup. The devising and analysis of the HHO discretization of the eigenvalue problem is the main subject of this work.

\subsection{Discrete setting}

Let $\K$ be a partition of $\Omega$ into non-overlapping mesh cells. A generic mesh cell is denoted by $K$ and can be a $d$-dimensional polytope with planar faces. In what follows, we assume that $\Omega$ is also a polytope in $\mathbb{R}^d$ with planar faces, so that the mesh can cover $\Omega$ exactly. For all $K\in\K$, we let $\bm{n}_K$ denote the unit outward vector to $K$. We say that $F\subset \mathbb{R}^d$ is a mesh face if it is a subset with nonempty relative interior of some affine hyperplane $H_F$ and if one of the two following conditions holds true: either there are two distinct mesh cells $K_1,K_2\in \K$ so that $F=\partial K_1\cap \partial K_2\cap H_F$ and $F$ is called an interface or there is one mesh cell $K\in\K$ so that $F=\partial K\cap \partial \Omega\cap H_F$ and $F$ is called a boundary face. The mesh faces are collected in the set $\F$, interfaces in the set $\F^{\mathrm{i}}$, and boundary faces in the set $\F^{\mathrm{b}}$. We let $h_S$ denote the diameter of the set $S$ which can be a mesh cell or a mesh face. We assume that the mesh $\K$ is a member of a shape-regular polytopal mesh family in the sense specified in \cite{DiPEL2014arbitrary,DiPEr2015hybridl}. In a nutshell, there is a matching simplicial submesh of $\K$ that belongs to a shape-regular family of simplicial meshes in the usual sense of Ciarlet \cite{ciarlet1978finite} and such that each cell $K\in \K$ (resp., face $F\in \F$) can be decomposed in a finite number of sub-cells (resp., sub-faces) with uniformly comparable diameter.

The HHO method is defined locally in each mesh cell $K \in \K$ from a pair of local unknowns which consist of one polynomial attached to the cell $K$ and a piecewise polynomial attached to the boundary $\partial K$, i.e., one polynomial attached to each face $F$ composing the boundary of $K$. Let $k \ge 0$ be a polynomial degree, and let $\mathbb{P}_{d'}^k(S)$, with $d'\in\{d-1,d\}$ be the linear space composed of real-valued polynomials of total degree at most $k$ on the $d'$-dimensional affine manifold $S\subset \mathbb{R}^d$ ($S$ is typically a mesh face or a mesh cell). The local discrete HHO pair is denoted 
\begin{equation} \label{eq:lsols}
\hat v_K = (v_K, v_{\partial K}) \in 
\hat V_K^k := \mathbb{P}_d^k(K) \times \mathbb{P}_{d-1}^k(\FK),
\end{equation}
where
\begin{equation}
\mathbb{P}_{d-1}^k(\FK) = \bigtimes_{F \in \FK} \mathbb{P}_{d-1}^k(F),
\end{equation} 
and $\FK$ is the collection of all the faces composing the cell boundary $\partial K$.
There is actually some flexibility in the choice of the polynomial degree for the cell unknowns since one can take them to be polynomials of degree $l\in\{k-1,k,k+1\}$ \cite{CoDPE2016bridging}. For simplicity, we only consider the case $l=k$; all what follows readily extends to the other choices for $l$. In what follows, we always use hat symbols to indicate discrete HHO pairs.

There are two key ingredients to devise locally the HHO method: a local reconstruction operator and a local stabilization operator.
The local reconstruction operator is defined as $p_K^{k+1}: \hat V_K^k \to \mathbb{P}_d^{k+1}(K)$ such that for all $\hat v_K = (v_K, v_{\partial K}) \in \hat V_K^k$, we have
\begin{equation} \label{eq:localnp}
(\nabla p_K^{k+1}(\hat v_K),\nabla w)_{L^2(K)} = (\nabla v_K,\nabla w)_{L^2(K)} + ( v_{\partial K} - v_K, \nabla w {\cdot} \bm{n}_{K})_{L^2(\partial K)}, 
\end{equation}
for all $w\in \mathbb{P}_d^{k+1}(K)$. The above Neumann problem uniquely defines $p_K^{k+1}  (\hat v_K) \in \mathbb{P}_d^{k+1}(K)$ up to an additive constant which can be specified by additionally requiring that $(p_K^{k+1}  (\hat v_K)-v_K,1)_{L^2(K)} = 0$ (this choice is irrelevant in what follows). 
Concerning stabilization, we define the local operator $S_{\partial K}^k : \hat V_K^k \to \mathbb{P}_{d-1}^k(\FK)$ such that, for all $\hat v_K = (v_K, v_{\partial K}) \in \hat V_K^k$, we have
\begin{equation} \label{eq:lbst}
S_{\partial K}^k(\hat v_K) = \Pi^k_{\partial K} (v_{\partial K}- p_K^{k+1}(\hat v_K)_{|\partial K} ) - \Pi_K^k(v_K-p_K^{k+1}(\hat v_K))_{|\partial K}, 
\end{equation}
where
 $\Pi^k_K$ and $\Pi_{\partial K}^k$ denote the $L^2$-orthogonal projectors from $L^1(K)$ onto $\mathbb{P}_d^k(K)$ and from $L^1(\dK)$ onto $\mathbb{P}_{d-1}^k(\FK)$, respectively. Equivalently, we have $S_{\partial K}^k(\hat v_K) = \Pi^k_{\partial K}(v_{\partial K}-P_K^{k+1}(\hat v_K)_{|\partial K})$ with $P_K^{k+1}(\hat v_K)=v_K+(I-\Pi_K^k)(p_K^{k+1}(\hat v_K))$, which is \cite[Eq.~(22)]{DiPEL2014arbitrary}.
Finally, the local HHO bilinear form for the stiffness is such that, for all $\hat v_K = (v_K, v_{\partial K}) \in \hat V_K^k$ and all $\hat w_K = (w_K, w_{\partial K}) \in \hat V_K^k$, we have
\begin{equation} \label{eq:lbs}
\hat a_K(\hat v_K, \hat w_K) = (\nabla p_K^{k+1}(\hat v_K),\nabla p_K^{k+1}(\hat w_K))_{L^2(K)} + (\tau_{\partial K}S_{\partial K}^k(\hat v_K),S_{\partial K}^k(\hat w_K))_{L^2(\partial K)},
\end{equation}
where $\tau_{\partial K}$ denotes the piecewise constant function on $\partial K$ such that $\tau_{\partial K|F}=\eta h_F^{-1}$ for all $F\in \FK$, and $\eta>0$ is a user-specified positive stabilization parameter (the simplest choice is to set $\eta=1$). 

\subsection{HHO discretization of the source problem}

To discretize the source problem~\eqref{eq:source_weak} using the HHO method, we consider the following global space of discrete HHO pairs:
\begin{equation} \label{eq:def_Vhk}
\hat V_h^k = V_{\K}^k \times V_{\F}^k,
\qquad
V^k_{\K} = \bigtimes_{K \in \K} \mathbb{P}_d^k(K),
\qquad
V^k_{\F} = \bigtimes_{F \in \F} \mathbb{P}_{d-1}^k(F).
\end{equation}
Here, the subscript $h$ refers to the global mesh-size defined as $h=\max_{K\in \K} h_K$.
For a global HHO pair $\hat v_h=(v_{\K},v_{\F}) \in \hat V^k_h$ with 
$v_{\K}\in V^k_{\K}$ and $v_{\F}\in V^k_{\F}$, we denote by 
$\hat v_K=(v_K,v_{\dK})\in\hat V_K^k$ the local HHO pair associated with the mesh cell $K\in\K$, and we denote by $v_F\in \mathbb{P}_{d-1}^k(F)$ the component associated with the mesh face $F\in \F$. The homogeneous Dirichlet boundary condition can be embedded into the HHO space by considering the subspaces 
\begin{equation} \label{eq:def_Vh0}
\hat V_{h,0}^k := V_\K^k \times V_{\F,0}^k, \qquad
V_{\F,0}^k := \{ v_\F \in V_{\F}^k \;|\; v_F = 0, \ \forall F\in \F^{\mathrm{b}} \}.
\end{equation} 
The HHO discretization of the source problem with $\phi\in L^2(\Omega)$ reads as follows: Find $\hat u_h \in \hat V_{h,0}^k$ such that
\begin{equation}
\label{eq:HHO_source}
\hat a_h(\hat u_h,\hat w_h) = b(\phi,w_{\K}), \qquad \forall \hat w_h=(w_{\K},w_{\F})\in \hat V_{h,0}^k,
\end{equation}
where
\begin{equation} \label{eq:def_hatah}
\hat a_h(\hat v_h,\hat w_h) = \sum_{K\in\K} \hat a_K(\hat v_K,\hat w_K), \qquad \forall \hat v_h,\hat w_h\in \hat V_h.
\end{equation}

The algebraic realization of the discrete source problem~\eqref{eq:HHO_source} leads to a symmetric linear system which can be written in the following block form where unknowns attached to the mesh cells are ordered before unknowns attached to the mesh faces: 
\begin{equation} \label{eq:m_source}
\begin{bmatrix}
\mathbf{A}_{\mathcal{KK}} & \mathbf{A}_{\mathcal{KF}} \\
\mathbf{A}_{\mathcal{FK}} & \mathbf{A}_{\mathcal{FF}} \\
\end{bmatrix}
\begin{bmatrix}
\mathbf{U}_\K \\
\mathbf{U}_\F\\
\end{bmatrix}
= 
\begin{bmatrix}
\bm{\phi}_{\K} \\ \mathbf{0}
\end{bmatrix}.
\end{equation}
The system matrix is positive-definite owing to the coercivity of the bilinear form $\hat a_h$; see~\eqref{eq:hatah_coer} below.
A computationally-effective way to solve the above linear system is to use a Schur complement technique, also known as static condensation, where the cell unknowns are eliminated by expressing them locally in terms of the face unknowns. This elimination is simple since the block-matrix $\mathbf{A}_{\mathcal{KK}}$ is block-diagonal. The resulting linear system in terms of the face unknowns is
\begin{equation} \label{eq:Schur_source}
\mathbf{K}_{\mathcal{FF}} \mathbf{U}_\F = - \mathbf{A}_{\mathcal{FK}}\mathbf{A}_{\mathcal{KK}}^{-1}\bm{\phi}_{\K},
\end{equation}
with the Schur complement matrix $\mathbf{K}_{\mathcal{FF}} = \mathbf{A}_{\mathcal{FF}}-\mathbf{A}_{\mathcal{FK}}\mathbf{A}_{\mathcal{KK}}^{-1}\mathbf{A}_{\mathcal{KF}}$.
As shown in~\cite{CoDPE2016bridging}, the linear system~\eqref{eq:Schur_source} is a global transmission problem (in which a given mesh face is locally coupled to the other mesh faces with which it shares a mesh cell) that expresses the equilibration of a suitable flux across all the mesh interfaces.

\subsection{HHO discretization of the eigenvalue problem}
The HHO discretization of the eigenvalue problem~\eqref{eq:eigen_weak} consists of finding the discrete eigenpairs $(\lambda_h,\hat u_h)\in \mathbb{R}_{>0} \times \hat V_{h,0}^k$ such that
\begin{equation} \label{eq:vfh}
\hat a_h(\hat u_h, \hat w_h) =  \lambda_h b(u_{\K},w_{\K}), \qquad \forall \hat w_h=(w_{\K},w_{\F})\in \hat V_{h,0}^k.
\end{equation}
One key idea here is that the mass bilinear form on the right-hand side of~\eqref{eq:vfh} only involves discrete cell unknowns.

The algebraic realization of~\eqref{eq:vfh} is the matrix eigenvalue problem
\begin{equation} \label{eq:mevp}
\begin{bmatrix}
\mathbf{A}_{\mathcal{KK}} & \mathbf{A}_{\mathcal{KF}} \\
\mathbf{A}_{\mathcal{FK}} & \mathbf{A}_{\mathcal{FF}} \\
\end{bmatrix}
\begin{bmatrix}
\mathbf{U}_\K \\
\mathbf{U}_\F\\
\end{bmatrix}
= \lambda_h
\begin{bmatrix}
\mathbf{B}_{\mathcal{KK}} & \mathbf{0}\\
\mathbf{0} & \mathbf{0} \\
\end{bmatrix}
\begin{bmatrix}
\mathbf{U}_\K \\
\mathbf{U}_\F\\
\end{bmatrix}.
\end{equation}
Since the face unknowns do not carry any mass, they can be eliminated, leading to the following matrix eigenvalue problem solely in terms of the cell unknowns:
\begin{equation}
\mathbf{K}_{\mathcal{KK}} \mathbf{U}_\K = \lambda_h \mathbf{B}_{\mathcal{KK}} \mathbf{U}_\K,
\end{equation}
with the Schur complement matrix $\mathbf{K}_{\mathcal{KK}} = \mathbf{A}_{\mathcal{KK}}-\mathbf{A}_{\mathcal{KF}}\mathbf{A}_{\mathcal{FF}}^{-1}\mathbf{A}_{\mathcal{FK}}$. Therefore, there are as many discrete eigenpairs as there are cell unknowns, i.e., the dimension of the polynomial space $\mathbb{P}_d^k$ times the number of mesh cells.

\subsection{HHO solution operators} 

We now introduce the key operators that play a central role in the analysis of the HHO approximation of the eigenvalue problem. To motivate the approach, we observe that for the source problem~\eqref{eq:HHO_source}, one can consider the cell-face HHO solution operator $\hat T_h : L^2(\Omega) \rightarrow \hat V_{h,0}^k$ so that 
\begin{equation} \label{eq:def_hatTh}
\hat a_h(\hat T_h(\phi),\hat w_h) = b(\phi,w_{\K}), \qquad \forall \hat w_h=(w_{\K},w_{\F})\in \hat V_{h,0}^k. 
\end{equation}
However, this operator is not convenient to analyze the approximation of the eigenvalue problem since it does not map to a subspace of $L^2(\Omega)$. The key idea is then to introduce a cell HHO solution operator $T_\K : L^2(\Omega) \rightarrow V_\K^k \subset L^2(\Omega)$  by mimicking the elimination of the face unknowns presented above at the algebraic level for the eigenvalue problem.  

As a first step, we define the operator $\ZF : V_\K^k \rightarrow V_{\F,0}^k$ so that, for all $v_\K\in V_\K^k$, $\ZF(v_\K)\in V_{\F,0}^k$ is defined as the unique solution of
\begin{equation} \label{eq:def_ZF}
\hat a_h((v_\K,\ZF(v_\K)),(0,w_\F)) = 0, \qquad \forall w_\F \in V_{\F,0}^k.
\end{equation}
To allow for some generality, we also define the operator $\ZFd : V_\K^k \rightarrow V_{\F,0}^k$ so that
\begin{equation} \label{eq:def_ZFd}
\hat a_h((0,w_\F),(v_\K,\ZFd(v_\K))) = 0, \qquad \forall w_\F \in V_{\F,0}^k.
\end{equation}
In the present setting where the bilinear form $\hat a_h$ is symmetric, the two operators $\ZF$ and $\ZFd$ coincide. As a second step, we define the bilinear form $a_\K$ on $V_\K^k\times V_\K^k$ such that
\begin{equation} \label{eq:def_aK}
a_\K(v_\K,w_\K) = \hat a_h((v_\K,\ZF(v_\K)),(w_\K,\ZFd(w_\K)),
\end{equation}
and introduce the solution operator $T_\K : L^2(\Omega) \rightarrow V_\K^k$ so that
\begin{equation} \label{eq:def_TK}
a_\K(T_\K(\phi),w_\K) =
\ b(\phi,w_K), \qquad \forall w_\K\in V_\K^k.
\end{equation}

\begin{lemma}[HHO solution operator] \label{lem:TK}
The following holds true:
\begin{equation}
\hat T_h(\phi) = (T_\K(\phi),(\ZF\circ T_\K)(\phi)), \qquad \forall \phi \in L^2(\Omega).
\end{equation}
\end{lemma}

\begin{proof}
Let $\phi \in L^2(\Omega)$. Let us set $u_\K = T_\K(\phi)$ so that $u_\K\in V_\K^k$ and $a_\K(u_\K,w_\K)=b(\phi,w_\K)$, for all $w_\K\in V_\K^k$, and set $u_\F = (\ZF\circ T_\K)(\phi) = \ZF(u_\K)$ so that $u_\F\in V_{\F,0}^k$. Setting $\hat u_h=(u_\K,u_\F) \in \hat V_{h,0}^k$, we need to verify that $\hat u_h$ solves the discrete HHO source problem, i.e.,
\[
\hat a_h(\hat u_h,\hat w_h) = b(\phi,w_\K), \qquad \forall \hat w_h=(w_{\K},w_{\F})\in \hat V_{h,0}^k. 
\]
Considering first a test function in the form $\hat w_h=(w_\K,0)$, we obtain
\begin{align*}
\hat a_h(\hat u_h,(w_\K,0)) &= \hat a_h((u_\K,\ZF(u_\K)),(w_\K,0)) \\
&= \hat a_h((u_\K,\ZF(u_\K)),(w_\K,0)) + \underbrace{\hat a_h((u_\K,\ZF(u_\K)),(0,\ZFd(w_\K)))}_{=0} \\
&= \hat a_h((u_\K,\ZF(u_\K)),(w_\K,\ZFd(w_\K))) \\
&= a_\K(u_\K,w_\K) = b(\phi,w_\K), 
\end{align*}
where we used the definition~\eqref{eq:def_ZF} of $\ZF$ in the second line and the definition~\eqref{eq:def_aK} of $a_\K$ in the fourth line.
Considering now a test function in the form $\hat w_h=(0,w_\F)$, we obtain owing to~\eqref{eq:def_ZF} that
\begin{align*}
\hat a_h(\hat u_h,(0,w_\F)) & = \hat a_h((u_\K,\ZF(u_\K)),(0,w_\F)) = 0.
\end{align*}
This completes the proof.
\end{proof}

The cell HHO solution operator $T_\K$ defined in~\eqref{eq:def_TK} is the relevant solution operator for the discrete eigenvalue problem~\eqref{eq:vfh}. Indeed, the eigenpair $(\lambda_h,\hat u_h)\in \mathbb{R}_{>0} \times \hat V_{h,0}^k$ with $\hat u_h=(u_\K,u_\F)\in V_\K^k\times V_{\F,0}^k$ solves~\eqref{eq:vfh} if and only if $u_\F=\ZF(u_\K)$ and the pair $(\lambda_h,u_\K)\in \mathbb{R}_{>0} \times V_\K^k$ solves
\begin{equation} \label{eq:celled_evp}
a_\K(u_\K,w_\K) = \lambda_h b(u_\K,w_\K), \qquad \forall w_\K \in V_\K^k,
\end{equation} 
that is, if and only if $(\mu_h,u_\K) \in \mathbb{R}_{>0} \times V_\K^k$ with $\mu_h = \lambda_h^{-1}$ is an eigenpair of the discrete solution operator $T_\K$.

\subsection{Error analysis for the source problem}

In this section we briefly outline the analysis of the HHO discretization of the source problem drawing on the ideas introduced in~\cite{DiPEL2014arbitrary}. One difference here is to include the case when the exact solution has a smoothness index $s\in (\frac12,k+2]$ and not just $s=k+2$ (recall that $s>\frac12$ follows from the elliptic regularity theory). In what follows, we use the symbol $C$ to denote a generic constant (its value can change at each occurrence) that can depend on the mesh regularity, the polynomial degree $k$ and the domain $\Omega$, but is independent of the mesh-size $h$. 

Let $K\in\K$ be a mesh cell. We equip the local HHO space $\hat V_K^k$ defined in~\eqref{eq:lsols} with the following seminorm (which is an HHO counterpart of the $H^1(K)$-seminorm)
\begin{equation} \label{eq:def_normK}
\|\hat v_K\|_{\hat V_K^k}^2 = \|\nabla v_K\|_{L^2(K)}^2 + \|\tau_{\dK}^{\frac12}(v_K-v_{\dK})\|_{L^2(\dK)}^2,
\end{equation}
for all $\hat v_K=(v_K,v_{\dK})\in \hat V_K^k$. We observe that $\|\hat v_K\|_{\hat V_K^k}=0$ implies that $v_K$ and $v_{\dK}$ are constant functions taking the same value. We equip the global HHO space $\hat V_h^k$ defined in~\eqref{eq:def_Vhk} with the seminorm (which is an HHO counterpart of the $H^1(\Omega)$-seminorm)
\begin{equation} \label{eq:def_normh}
\|\hat v_h\|_{\hat V_h^k}^2 = \sum_{K\in\K} \|\hat v_K\|_{\hat V_K^k}^2, \qquad \forall \hat v_h\in \hat V_h^k.
\end{equation}
The map $\|\cdot\|_{\hat V_h^k}$ is a norm on the subspace $\hat V_{h,0}^k$ defined in~\eqref{eq:def_Vh0}. \cite[Lemma~4]{DiPEL2014arbitrary} shows that there is a real number $\beta>0$, uniform with respect to the mesh-size $h$, such that, for all $K\in \K$,
\begin{equation} \label{eq:norm_equiv}
\beta \|\hat v_K\|_{\hat V_K^k}^2 \le \hat a_K(\hat v_K,\hat v_K) \le \beta^{-1} \|\hat v_K\|_{\hat V_K^k}^2, \qquad \forall \hat v_K \in \hat V_K^k,
\end{equation}
and, consequently, given the definition~\eqref{eq:def_hatah} of $\hat a_h$, that the following coercivity and boundedness properties hold true:
\begin{alignat}{2} 
\hat a_h(\hat v_h,\hat v_h) &\ge \beta \|\hat v_h\|_{\hat V_h^k}^2, &\qquad&\forall \hat v_h\in\hat V_h^k, \label{eq:hatah_coer}\\
\hat a_h(\hat v_h,\hat w_h) &\le \beta^{-1} \|\hat v_h\|_{\hat V_h^k}\|\hat w_h\|_{\hat V_h^k}, &\qquad&\forall (\hat v_h,\hat w_h)\in\hat V_h^k\times \hat V_h^k. \label{eq:hatah_bnd}
\end{alignat}
Owing to the Lax--Milgram Lemma, we infer that the cell-face HHO solution operator $\hat T_h : L^2(\Omega) \rightarrow \hat V_{h,0}^k$ introduced in~\eqref{eq:def_hatTh} is well-defined. For later use in the analysis of the eigenvalue problem, we now establish a stability property for $\hat T_h$.

\begin{lemma}[Stability of $\hat T_h$] \label{lem:stab_hatTh}
There is $C$ so that 
\begin{equation}
\|\hat T_h(\phi)\|_{\hat V_h^k} \le C \|\phi\|_{L^2(\Omega)}, \qquad
\forall \phi \in L^2(\Omega).
\end{equation}
\end{lemma}

\begin{proof}
Let $\phi \in L^2(\Omega)$ and let us write $\hat T_h(\phi)=(u_\K,u_\F)$ with $u_\K\in V_\K^k$ and $u_\F\in V_{\F,0}^k$. 
Using the coercivity property~\eqref{eq:hatah_coer}, the definition~\eqref{eq:def_hatTh} of the solution operator $\hat T_h$, that of the bilinear form $b$, and the Cauchy--Schwarz inequality leads to
\begin{equation} \label{eq:coerc_CS}
\beta \|\hat T_h(\phi)\|_{\hat V_h^k}^2 \le \hat a_h(\hat T_h(\phi),\hat T_h(\phi)) = b(\phi,u_\K) \le \|\phi\|_{L^2(\Omega)} \|u_\K\|_{L^2(\Omega)}.
\end{equation}
On the broken polynomial space $V_\K^k$, we can apply the following discrete Poincar\'e inequality which has been derived in the discontinuous Galerkin context in~\cite{Arnol:82,Brenn:03,DiPEr:10}:
\[
C_{\mathrm{P,dG}} \|u_\K\|_{L^2(\Omega)} \le \bigg( \sum_{K\in\K} \|\nabla u_K\|_{L^2(K)}^2 + \sum_{F\in \F} h_F^{-1}\|[\!\![ u_\K ]\!\!]_F\|_{L^2(F)}^2 \bigg)^{\frac12},
\]
with $C_{\mathrm{P,dG}}>0$ uniform with respect to the mesh-size $h$, and where $[\!\![ u_\K ]\!\!]_F$ denotes the jump of $u_\K$ across $F$ if $F$ is an interface ($F\in\F^{\mathrm{i}}$) or the value of $u_\K$ on $F$ if $F$ is a boundary face ($F\in\F^{\mathrm{b}}$). If $F\in\F^{\mathrm{i}}$, we have $[\!\![ u_\K ]\!\!]_F=u_{K_1}|_F-u_{K_2}|_F$ where $K_1,K_2$ are the two mesh cells sharing $F$ (the sign of the jump is irrelevant in what follows), and we can therefore write $[\!\![ u_\K ]\!\!]_F = (u_{K_1}-u_F)|_F - (u_{K_2}-u_F)|_F$ where $u_F$ is the component of $u_\F$ attached to $F$. If $F\in \F^{\mathrm{b}}$, we have $[\!\![ u_\K ]\!\!]_F=u_{K_1}|_F$ where $K_1$ is the unique mesh cell sharing $F$ with $\partial\Omega$, and we can therefore write $[\!\![ u_\K ]\!\!]_F=(u_{K_1}-u_F)|_F$ since $u_F\equiv0$ (recall that $u_\F\in V_{\F,0}^k$). Recalling the definition~\eqref{eq:def_normh} of the $\|\cdot\|_{\hat V_h^k}$-norm, that of $\tau_{\dK}$ given just below~\eqref{eq:lbs}, and using the triangle inequality, we infer that
\[
\|u_\K\|_{L^2(\Omega)} \le C \|(u_\K,u_\F)\|_{\hat V_h^k} = C\|\hat T_h(\phi)\|_{\hat V_h^k}.
\]
Combining this bound with~\eqref{eq:coerc_CS}, we obtain the assertion.
\end{proof}

An important tool in the analysis of HHO methods is the global reduction operator $\hat I_h^k : H^1_0(\Omega) \rightarrow \hat V_{h,0}^k$ defined such that, for all $v\in H^1_0(\Omega)$,
\begin{equation} \label{eq:def_hatIh}
\hat I_h^k(v) = (\Pi_\K^k(v),\Pi_\F^k(v)) := ((\Pi_K^k(v))_{K\in\K},(\Pi_F^k(v))_{F\in\F}) \in \hat V_{h,0}^k,
\end{equation}
where $\Pi_K^k$ and $\Pi_F^k$ denote the $L^2$-orthogonal projectors onto $\mathbb{P}_d^k(K)$ and $\mathbb{P}_{d-1}^k(F)$, respectively. We also define the local reduction operator $\hat I_K^k:  \ H^1(K) \rightarrow \hat V_K^k$ such that, for all $v \in H^1(K)$, 
\begin{equation}
\hat I_K^k(v)=(\Pi_K^k(v),\Pi_{\dK}^k(v)) = (\Pi_K^k(v),(\Pi_F^k(v))_{F\in\FK}) \in \hat V_K^k.
\end{equation}
Recalling the local reconstruction operator $p_K^{k+1}: \hat V_K^k \to \mathbb{P}_d^{k+1}(K)$ defined in~\eqref{eq:localnp}, \cite[Lemma~3]{DiPEL2014arbitrary} shows that
\begin{equation} \label{eq:elliptic_proj}
e_K^{k+1} := p_K^{k+1}\circ \hat I_K^k : H^1(K) \rightarrow \mathbb{P}_d^{k+1}(K),
\end{equation}
is the elliptic projector, i.e., for all $v\in H^1(K)$, $e_K^{k+1}(v)$ is the unique polynomial in $\mathbb{P}_d^{k+1}(K)$ such that $(\nabla (e_K^{k+1}(v)-v),\nabla w)_{L^2(K)} = 0$ for all $w\in \mathbb{P}_d^{k+1}(K)$ and $(e_K^{k+1}(v)-v,1)_{L^2(K)}=0$. For two functions $v,w\in H^1(K)$, the above orthogonality condition on the gradient implies that
\begin{equation} \label{eq:Pytha_ell}
(\nabla(e_K^{k+1}(v)-v), \nabla( e_K^{k+1}(w)-w))_{L^2(K)} = (\nabla v,\nabla w)_{L^2(K)} - (\nabla e_K^{k+1}(v),\nabla e_K^{k+1}(w))_{L^2(K)}.
\end{equation}

\begin{lemma}[Discrete error estimate] \label{lem:err_est}
There is $C$ such that
\begin{equation}
\|\hat T_h(\phi)- \hat I_h^k(T(\phi))\|_{\hat V_h^k} \le Ch^t\|T(\phi)\|_{H^{1+t}(\Omega)},
\end{equation}
for all $t\in [s,k+1]$, and all $\phi\in L^2(\Omega)$ such that $T(\phi)\in H^{1+t}(\Omega)$; here, $s>\frac12$ is the smoothness index resulting from the elliptic regularity theory.
\end{lemma}

\begin{proof}
Let $t\in [s,k+1]$, and let $\phi\in L^2(\Omega)$ be such that $T(\phi)\in H^{1+t}(\Omega)$. Proceeding as in the proof of \cite[Theorem~8]{DiPEL2014arbitrary}, we infer that
\[
\|\hat T_h(\phi)- \hat I_h^k(T(\phi))\|_{\hat V_h^k} \le C \sup_{\substack{\hat w_h\in \hat V_{h,0}^k \\\|\hat w_h\|_{\hat V_h^k}=1}} |\delta_h(\hat w_h)| =: C\, \|\delta_h\|_{(\hat V_{h,0}^k)'},
\]
with the consistency error $\delta_h(\hat w_h)$ such that
\begin{align*}
\delta_h(\hat w_h) = {}&\sum_{K\in\K} (\nabla \xi_K,\nabla w_K)_{L^2(K)} + (\nabla\xi_K{\cdot}\bm{n}_{K},w_{\dK}-w_K)_{L^2(\dK)} \\ &+ (\tau_{\dK}S_{\dK}^k(\hat I_K^k(u)),S_{\dK}^k(\hat w_K))_{L^2(\dK)},
\end{align*}
and the shorthand notation $\xi_K:=e_K^{k+1}(u_{|K})-u_{|K}$ and $u=T(\phi)$ (we used $s>\frac12$ in writing the second summand on the right-hand side above). Using the Cauchy--Schwarz inequality and recalling the definition of the norm $\|\hat w_h\|_{\hat V_h^k}$, we obtain 
\[
\|\delta_h\|_{(\hat V_{h,0}^k)'} \le C\, \bigg(\sum_{K\in\K} \|\nabla\xi_K\|_{L^2(K)}^2 + h_K\|\nabla\xi_K\|_{L^2(\dK)}^2 + h_K^{-1}\|S_{\dK}^k(\hat I_K^k(u))\|_{L^2(\dK)}^2\bigg)^{\frac12}.
\]
Recalling the definition~\eqref{eq:lbst} of the stabilization operator $S_{\dK}^k$, 
we obtain that
\begin{align*}
S_{\dK}^k(\hat I_K^k(u)) &= \Pi_{\dK}^k(\Pi_{\dK}^k(u)-e_K^{k+1}(u)_{|\partial K})-\Pi_K^k(\Pi_{\dK}^k(u)-e_K^{k+1}(u))_{|\partial K} \\
&= \Pi_{\dK}^k((u-e_K^{k+1}(u))_{|\partial K})-\Pi_K^k(u-e_K^{k+1}(u))_{|\partial K} \\
&= -\Pi_{\dK}^k((\xi_K)_{|\partial K}) + \Pi_K^k(\xi_K)_{|\partial K}.
\end{align*}
We then have
\begin{align*}
\|S_{\dK}^k(\hat I_K^k(u))\|_{L^2(\dK)} &\le \|\xi_K\|_{L^2(\dK)} + \|\Pi_K^k(\xi_K)\|_{L^2(\dK)} \\
&\le \|\xi_K\|_{L^2(\dK)} + Ch_K^{-\frac12}\|\xi_K\|_{L^2(K)} \\
&\le C(h_K^{-\frac12}\|\xi_K\|_{L^2(K)}+h_K^{\frac12}\|\nabla \xi_K\|_{L^2(K)}) \\
&\le Ch_K^{\frac12}\|\nabla \xi_K\|_{L^2(K)},
\end{align*}
where we used a triangle inequality and the $L^2$-stability of $\Pi_{\dK}^k$ in the first line, a discrete trace inequality and the $L^2$-stability of $\Pi_K^k$ in the second line, a multiplicative trace inequality in the third line, and the Poincar\'e--Steklov inequality on $K$ in the fourth line (that is, $\|\xi_K\|_{L^2(K)}\le Ch_K\|\nabla \xi_K\|_{L^2(K)}$ since $\xi_K$ has zero mean-value in $K$ by construction).
We conclude that
$h_K^{-\frac12}\|S_{\dK}^k(\hat I_K^k(u))\|_{L^2(\dK)}\le C\|\nabla\xi_K\|_{L^2(K)}$ for some generic constant $C$,
and therefore, we have
\[
\|\delta_h\|_{(\hat V_{h,0}^k)'} \le C\, \bigg(\sum_{K\in\K} \|\nabla\xi_K\|_{L^2(K)}^2 + h_K\|\nabla\xi_K\|_{L^2(\dK)}^2\bigg)^{\frac12}.
\]
Finally, invoking the approximation properties of the elliptic projector on all the mesh cells leads to the assertion.
\end{proof}

\section{Error analysis for the eigenvalue problem} 
\label{sec:ea}

The goal of this section is to perform the error analysis of the discrete eigenvalue problem~\eqref{eq:vfh} by using the abstract theory outlined in Section~\ref{sec:theory} in the Hilbert space $L=L^2(\Omega)$. Let $T,T^* : L^2(\Omega) \rightarrow H^1_0(\Omega)\subset L^2(\Omega)$ be the exact solution and adjoint solution operators defined in Section~\ref{sec:statement} ($T=T^*$, i.e., $T$ is selfadjoint, in the present symmetric setting). Let $T_\K : L \rightarrow V_\K^k\subset L$ be the discrete HHO solution operator defined in~\eqref{eq:def_TK}. Its adjoint operator $T_\K^* : L \rightarrow V_\K\subset L$ is defined so that, for all $\psi \in L$, $T_\K^*(\psi)\in V_\K$ is the unique solution of
\begin{equation} \label{eq:def_TKadj}
a_\K(w_\K,T_\K^*(\psi)) = b(w_\K,\psi), \qquad \forall w_\K\in V_\K^k.
\end{equation} 
Owing to the symmetry of the bilinear forms $a_\K$ and $b$, we have $T_\K=T_\K^*$ in the present setting, i.e., $T_\K$ is selfadjoint. We keep as before a distinct notation to allow for more generality, and we also set $\hat T_h^\dagger : L \to \hat V_{h,0}^k$ so that $\hat T_h^\dagger(\psi)=(T_\K^*(\psi),(\ZFd\circ T_\K^*)(\psi))$ for all $\psi \in L$. Proceeding as in Lemma~\ref{lem:TK}, we conclude that 
\begin{equation} \label{eq:hatThdag}
\hat a_h(\hat w_h,\hat T_h^\dagger(\psi)) = b(w_\K,\psi), \qquad \forall \hat w_h=(w_{\K},w_{\F})\in \hat V_{h,0}^k. 
\end{equation}
In the present symmetric setting, we have $\hat T_h=\hat T_h^\dagger$
with $\hat T_h$ defined in~\eqref{eq:def_hatTh}.
Finally, the elliptic regularity theory
implies that there is a real number $s\in (\frac12,1]$ so that 
$T,T^* \in \mathcal{L}(L^2(\Omega);H^{1+s}(\Omega))$, with 
operator norm denoted by $C_s$.

\subsection{Preliminary results}
To verify that we can apply the abstract theory from Section~\ref{sec:theory}, let us show that $T_\K$ converges to $T$ in operator norm as the mesh-size $h$ tends to zero, i.e., that \eqref{eq:normc} holds true.

\begin{lemma}[Bound on $L\times L$] \label{lem:bnd0}
The following holds true: 
\begin{equation}
\sup_{(\phi,\psi)\in L\times L} |((T-T_\K)(\phi),\psi)_L| \le C h^{s} \|\phi \|_L  \| \psi \|_L.
\end{equation}
where $s\in (\frac12,1]$ is the smoothness index associated with the elliptic regularity theory.
Consequently, we have $\|T-T_\K\|_{\mathcal{L}(L;L)}\to 0$ as $h\to 0$.
\end{lemma}

\begin{proof}
For all $\phi,\psi \in L$, we have
\begin{align}
&((T-T_\K)(\phi),\psi)_{L}\nonumber \\ &= (T(\phi),\psi)_{L} - b(T_\K(\phi),\psi) \nonumber\\
&= (T(\phi),\psi)_{L} - a_\K(T_\K(\phi),T_\K^*(\psi)) \nonumber\\
&= (T(\phi),\psi)_{L} - \hat a_h(\hat T_h(\phi),\hat T_h^\dagger(\psi)) \nonumber\\
&= (T(\phi),\psi)_{L} - \hat a_h(\hat I_h^k(T(\phi)),\hat T_h^\dagger(\psi)) + \hat a_h(\hat I_h^k(T(\phi))-\hat T_h(\phi),\hat T_h^\dagger(\psi)) \nonumber\\
&= (T(\phi)-\Pi_\K^k(T(\phi)),\psi)_{L} + \hat a_h(\hat I_h^k(T(\phi))-\hat T_h(\phi),\hat T_h^\dagger(\psi)), \label{eq:calcul1}
\end{align}
where we used the definition of the bilinear form $b$ in the first line, the definition~\eqref{eq:def_TKadj} of $T_\K^*$ in the second line, the definition~\eqref{eq:def_aK} of $a_\K$ and Lemma~\ref{lem:TK} in the third line, a simple algebraic manipulation in the fourth line, and the property~\eqref{eq:hatThdag} and the definition~\eqref{eq:def_hatIh} of $\hat I_h^k$ in the fifth line. 
Let us call $S_1,S_2$ the two summands on the right-hand side of~\eqref{eq:calcul1}. Owing to the elliptic regularity theory and the approximation properties of the projector $\Pi_\K^k$ (with $k\ge0$), we obtain that
\[
|S_1| \le Ch|T(\phi)|_{H^1(\Omega)}\|\psi\|_L.
\]
Since $|T(\phi)|_{H^1(\Omega)} \le \|T(\phi)\|_{H^{1+s}(\Omega)} \le C_s\|\phi\|_L$, we infer that
\[
|S_1| \le CC_sh \|\phi\|_L\|\psi\|_L.
\]
To bound $S_2$, we use the boundedness property~\eqref{eq:hatah_bnd} of $\hat a_h$ followed by the error estimate from Lemma~\ref{lem:err_est} (with $t=s$) and the stability property of $\hat T_h^\dagger=\hat T_h$ from Lemma~\ref{lem:stab_hatTh} to infer that
\[
|S_2| \le Ch^s\|T(\phi)\|_{H^{1+s}(\Omega)}\|\psi\|_L \le CC_sh^s\|\phi\|_L\|\psi\|_L.
\]  
Combining the bounds on $S_1$ and $S_2$  
concludes the proof.
\end{proof}

Let $\mu \in \sigma(T)\setminus\{0\}$ with ascent $\alpha$ 
and algebraic multiplicity $m$.
To quantify the smoothness of the functions in the subspaces $G_\mu$ and 
$G_\mu^*$ defined in~\eqref{eq:def_Gmu}, we assume that there is a real number
$t\in [s,k+1]$ and a constant $C_t$ so that 
\begin{equation}
\label{eq:smoothness_t}
\begin{alignedat}{2}
\|\phi\|_{H^{1+t}(\Omega)} + \|T(\phi)\|_{H^{1+t}(\Omega)} & \le C_t\|\phi\|_L, &\qquad& \forall \phi \in G_\mu, \\
\|\psi\|_{H^{1+t}(\Omega)} + \|T^*(\psi)\|_{H^{1+t}(\Omega)} & \le C_t\|\psi\|_L, &\qquad&  \forall \psi\in G_\mu^*.
\end{alignedat}
\end{equation} 
Note that $t$ depends on $\mu$, but we just write $t$ instead of $t_\mu$ to alleviate the notation. If $t=s$, functions in $G_\mu$ and $G_\mu^*$ do not provide additional smoothness with respect to that resulting from the elliptic regularity theory. In general, functions in $G_\mu$ and $G_\mu^*$ are smoother, and one has $t>s$. The case $t=k+1$ leads to optimal error estimates, see Remark~\ref{rem:optimal} below.

\begin{lemma}[Bound on $G_\mu\times L$ and $L\times G_\mu^*$] \label{lem:bnd1}
The following holds true:
\begin{equation}
\sup_{(\phi,\psi)\in G_\mu\times L} |((T-T_\K)(\phi),\psi)_L| \le C h^{t} \|\phi \|_L  \| \psi \|_L,
\end{equation}
where $t\in [s,k+1]$ is the smoothness index associated with $\mu$.
Consequently, we have
\begin{equation}
\|(T-T_\K)|_{G_\mu} \|_{\mathcal{L}(G_\mu;L)} \le Ch^{t}.
\end{equation}
Similar bounds hold for $T_\K^*$, and in particular, we have
$\|(T-T_\K)^*|_{G_\mu^*} \|_{\mathcal{L}(G_\mu^*;L)} \le Ch^{t}$.
\end{lemma}

\begin{proof}
We only prove prove the statement for $T_\K$, the other proof is similar. Our starting point is~\eqref{eq:calcul1}. Owing to the smoothness of the function $T(\phi)$ resulting from~\eqref{eq:smoothness_t}, we infer that
\[
|S_1| \le Ch^{\min(k+1,t+1)}\|T(\phi)\|_{H^{1+t}(\Omega)}\|\psi\|_L
\le CC_th^{\min(k+1,t+1)}\|\phi\|_L\|\psi\|_L.
\]
Using similar arguments leads to $|S_2| \le Ch^t\|\phi\|_L\|\psi\|_L$ and since $t\le \min(k+1,t+1)$, the assertion follows.
\end{proof}

\begin{lemma}[Bound on $G_\mu\times G_\mu^*$] \label{lem:bnd2}
The following holds true:
\begin{equation}
\sup_{(\phi,\psi)\in G_\mu\times G_\mu^*} |((T-T_\K)(\phi),\psi)_L| \le C h^{2t} \|\phi \|_L  \| \psi \|_L,
\end{equation}
where $t\in [s,k+1]$ is the smoothness index associated with $\mu$.
\end{lemma}

\begin{proof}
Our starting point is again~\eqref{eq:calcul1}, but we can now derive sharper bounds on the two summands $S_1$ and $S_2$ by exploiting the smoothness of both $\phi$ and $\psi$. On the one hand, we have
\[
S_1 = (T(\phi)-\Pi_\K^k(T(\phi)),\psi)_{L} = (T(\phi)-\Pi_\K^k(T(\phi)),\psi-\Pi_\K^k(\psi))_{L},
\]
so that 
\[
|S_1| \le Ch^{2\min(k+1,t+1)}\|T(\phi)\|_{H^{1+t}(\Omega)}\|\psi\|_{H^{1+t}(\Omega)}
\le CC_t^2h^{2\min(k+1,t+1)}\|\phi\|_L\|\psi\|_L,
\]
where we used the smoothness of the functions $T(\phi)$ and $\psi$ resulting from~\eqref{eq:smoothness_t}. On the other hand, we have
\begin{align*}
S_2 ={}& \hat a_h(\hat I_h^k(T(\phi))-\hat T_h(\phi),\hat T_h^\dagger(\psi)) \\
={}& \hat a_h(\hat I_h^k(T(\phi))-\hat T_h(\phi),\hat I_h^k(T^*(\psi))) + \hat a_h(\hat I_h^k(T(\phi))-\hat T_h(\phi),\hat T_h^\dagger(\psi)-\hat I_h^k(T^*(\psi))) \\
={}&  a(T(\phi),T^*(\psi)) - \hat a_h(\hat T_h(\phi),\hat I_h^k(T^*(\psi))) \\
&+ \hat a_h(\hat I_h^k(T(\phi)),\hat I_h^k(T^*(\psi)))-a(T(\phi),T^*(\psi)) \\
&+ \hat a_h(\hat I_h^k(T(\phi))-\hat T_h(\phi),\hat T_h^\dagger(\psi)-\hat I_h^k(T^*(\psi))) \\
={}& (\phi-\Pi_\K^k(\phi),T^*(\psi)-\Pi_\K^k(T^*(\psi)))_L \\
&+ \hat a_h(\hat I_h^k(T(\phi)),\hat I_h^k(T^*(\psi)))-a(T(\phi),T^*(\psi)) \\
&+ \hat a_h(\hat I_h^k(T(\phi))-\hat T_h(\phi),\hat T_h^\dagger(\psi)-\hat I_h^k(T^*(\psi)))
\end{align*}
where we used simple algebraic manipulations to derive the second and third identities, and the definition of $T$ together with that of $\hat T_h$ and of $\hat I_h^k$ to derive the last identity. Let us call $S_{2,1}$, $S_{2,2}$, $S_{2,3}$ the three summands on the right-hand side of the above equation. Reasoning as above and invoking the smoothness of the functions $\phi$ and $T^*(\psi)$ resulting from~\eqref{eq:smoothness_t}, we infer that
\[
|S_{2,1}| \le CC_t^2h^{2\min(k+1,t+1)}\|\phi\|_L\|\psi\|_L.
\]  
To bound $S_{2,2}$, we observe that
\begin{align*}
S_{2,2} = {}&\sum_{K\in\K} (\nabla e_K^{k+1}(T(\phi)),\nabla e_K^{k+1}(T^*(\psi)))_{L^2(K)} - (\nabla T(\phi),\nabla T^*(\psi))_{L^2(K)} \\ &+ \sum_{K\in\K} (\tau_{\dK}S_{\dK}^k(\hat I_h^k(T(\phi))),S_{\dK}^k(\hat I_h^k(T^*(\phi))))_{L^2(\dK)} =: S_{2,2,1} + S_{2,2,2}.
\end{align*}
Since $e_K^{k+1}$ is the elliptic projector, the identity~\eqref{eq:Pytha_ell} implies that
\[
S_{2,2,1} = \sum_{K\in\K} -(\nabla (T(\phi)-e_K^{k+1}(T(\phi))),\nabla (T^*(\psi)-e_K^{k+1}(T^*(\psi))))_{L^2(K)},
\]
Using the Cauchy--Schwarz inequality and the approximation properties of the elliptic projector, we infer that
\[
|S_{2,2,1}| \le Ch^{2t}\|T(\phi)\|_{H^{1+t}(\Omega)} \|T^*(\psi)\|_{H^{1+t}(\Omega)} \le CC_t^2h^{2t} \|\phi\|_L\|\psi\|_L.
\]
Moreover, reasoning as in the end of the proof of Lemma~\ref{lem:err_est}, we obtain that
\[
|S_{2,2,2}| \le Ch^{2t} \|T(\phi)\|_{H^{1+t}(\Omega)} \|T^*(\psi)\|_{H^{1+t}(\Omega)} \le CC_t^2h^{2t} \|\phi\|_L\|\psi\|_L.
\]
Hence, we have
\[
|S_{2,2}| \le CC_t^2h^{2t} \|\phi\|_L\|\psi\|_L.
\]
Finally, the bound on $S_{2,3}$ results from the boundedness property~\eqref{eq:hatah_bnd} of $\hat a_h$ and the error estimate from Lemma~\ref{lem:err_est} since
\[
|S_{2,3}| \le Ch^{2t}\|T(\phi)\|_{H^{1+t}(\Omega)} \|T^*(\psi)\|_{H^{1+t}(\Omega)} \le CC_t^2h^{2t}\|\phi\|_L\|\psi\|_L.
\]
Collecting the above estimates concludes the proof.
\end{proof}

\subsection{Main results}
We can now present our main results. 
Let $\mu \in \sigma(T)\setminus\{0\}$ with ascent $\alpha$ 
and algebraic multiplicity $m$. We focus now on the spectral approximation of selfadjoint operators, so that we have $\alpha=1$.  Owing to the convergence result from Lemma~\ref{lem:bnd0}, there are $m$ eigenvalues of $T_\K$, denoted $\mu_{h,1},\ldots,\mu_{h,m}$, that converge to $\mu$ as $h\to 0$.

\begin{theorem}[Error estimate on eigenvalues and eigenfunctions in $L$] \label{thm:errest}
Assume that there is $t\in [s,k+1]$ so that the smoothness property~\eqref{eq:smoothness_t} holds true, where $s>\frac12$ is the smoothness index resulting from the elliptic regularity theory. Then
there is $C$, depending on $\mu$ (and on the mesh regularity, the polynomial degree $k$ and the domain $\Omega$) but independent of the mesh-size $h$, such that 
\begin{equation} \label{eq:esterr_eve}
\max_{1 \le j \le m} |\mu - \mu_{h,j}| \le C h^{2t}.
\end{equation}
Furthermore, let $u_{\K,j}\in V_\K^k$ be a unit vector in $\text{\em ker}(\mu_{h,j} I- T_\K)$. Then,
there is a unit vector $u_j \in \text{\em ker}(\mu I-T) \subset G_\mu$ such that
\begin{equation} \label{eq:esterr_efe}
\| u_{j} - u_{\K,j} \|_L \le C h^{t}.
\end{equation}
\end{theorem}

\begin{proof}
Combining the results from Lemma~\ref{lem:bnd1}, and Lemma~\ref{lem:bnd2} with Theorem~\ref{thm:eve0} and Theorem~\ref{thm:efe0} completes the proof. 
\end{proof}

\begin{remark}[Error estimate on eigenvalues]
Since the eigenvalues $\lambda$ and $\lambda_h$ associated with~\eqref{eq:eigen_weak} and~\eqref{eq:vfh}, respectively, are such that $\lambda=\mu^{-1}$ and $\lambda_h=\mu_h^{-1}$, we infer that the same estimate as~\eqref{eq:esterr_eve} holds true for the error between $\lambda$ and $\lambda_h$.
\end{remark}

\begin{corollary}[Eigenfunction error estimate in $H^1$] \label{cor:H1}
Let us drop the index $j$ for simplicity from the
eigenfunction $u_j$ and the approximate eigenfunction $u_{\K,j}$ and let us set $\hat u_h=(u_\K,\ZF(u_\K))$.
Then the following holds true:
\begin{equation} \label{eq:eigenfunction_hatah}
 \hat a_h(\hat u_h- \hat I_h^k(u),\hat u_h - \hat I_h^k(u))^{\frac12} \le Ch^{t}.
\end{equation}
Consequently, we have
\begin{equation} \label{eq:eigenfunction_H1}
\bigg( \sum_{K\in\K} \|\nabla(u-p_K^{k+1}(\hat u_K))\|_{L^2(K)}^2 \bigg)^{\frac12} \le Ch^{t}.
\end{equation}
\end{corollary}

\begin{proof}
We observe that
\begin{align*}
\lambda_h (u_\K,u)_{L}  ={}& \lambda_h (u_\K,\Pi_\K^k(u))_{L} = \lambda_h b(u_\K,\Pi_\K^k(u)) = a_\K (u_\K,\Pi_\K^k(u)) \\
={}& \hat a_h((u_\K,\ZF(u_\K)), (\Pi_\K^k(u),\ZFd(\Pi_\K^k(u))) \\
={}& \hat a_h((u_\K,\ZF(u_\K)), (\Pi_\K^k(u),\ZFd(\Pi_\K^k(u))) \\
& + \hat a_h((u_\K,\ZF(u_\K)),(0, \Pi_{\F}^k(u) - \ZFd(\Pi_\K^k(u))) \\
={}& \hat a_h( \hat u_h, \hat I_h^k(u) ),
\end{align*}
where we have used the definition of $\Pi_\K^{k}$ and \eqref{eq:celled_evp} in the first line, the definition \eqref{eq:def_aK} of $\hat a_h$ in the second line, the property \eqref{eq:def_ZF} of $Z_{\F,0}$ in the third line, and the definition of $\hat I_h^k$ in the last line.
Setting $\delta_u := \hat a_h(\hat I_h^k(u), \hat I_h^k(u)) - a(u,u)$ and recalling the normalization $\| u \|_{L} = \| u_\K \|_{L} = 1$, we infer that
\begin{align*}
\hat a_h(\hat u_h- \hat I_h^k(u),\hat u_h - \hat I_h^k(u)) 
&= \hat a_h(\hat u_h,\hat u_h) - 2\hat a_h(\hat u_h, \hat I_h^k(u)) + \hat a_h(\hat I_h^k(u), \hat I_h^k(u))  \\
&= \lambda_h \| u_\K \|^2_{L} - 2\lambda_h(u_\K,u)_{L} + \lambda_h \| u \|^2_{L} - (\lambda_h - \lambda) \| u \|^2_{L} + \delta_u  \\
& = \lambda_h\| u_\K-u\|_{L}^2 - \lambda_h + \lambda +  \delta_u,
\end{align*}
which is a generalization of the Pythagorean eigenvalue error identity (see \cite{strang1973analysis}) in the HHO context. The bound~\eqref{eq:eigenfunction_hatah} then follows from the bounds derived in Theorem~\ref{thm:errest} (see in particular the bound on $S_{2,2}$ therein to estimate $\delta_u$). 
Finally, the bound~\eqref{eq:eigenfunction_H1} follows from the definition of the bilinear form $\hat a_h$,  
the triangle inequality, and the approximation properties of the elliptic projector.
\end{proof}

\begin{remark}[Optimal convergence] \label{rem:optimal}
If $t=k+1$, we recover a convergence of order $h^{2k+2}$ for the eigenvalues and of order $k^{k+1}$ for the eigenfunctions in the $H^1$-seminorm. 
\end{remark}

\section{Numerical experiments} \label{sec:num}
In this section, we first verify the error estimates from Section \ref{sec:ea} for eigenvalues and smooth eigenfunctions approximated by the HHO method in 1D (unit interval) and in 2D (unit square).
We then study the effect of varying the stabilization parameter and, in particular, we report superconvergence results for 1D uniform meshes when using a particular value of the stabilization parameter. We next consider in 2D the use of polygonal (hexagonal) meshes and we compare our results to those obtained using continuous finite elements. Finally, we present convergence results on an L-shaped domain (which includes the case of a non-smooth eigenfunction) and on the unit disk. 
In all cases, we consider the eigenvalues $\lambda$ and $\lambda_h$ associated with~\eqref{eq:eigen_weak} and~\eqref{eq:vfh}, respectively; both sets of eigenvalues are sorted in an increasing order as $\lambda_1< \lambda_2 \ldots$ and $\lambda_{1,h}<\lambda_{2,h}\ldots$, and we report the normalized eigenvalue errors
$\frac{|\lambda_j - \lambda_{h,j}|}{\lambda_j}$.

\ifNUM

\begin{table}[ht]
\centering 
\begin{tabular}{| c | c || cc | cc | cc | cc |}
\hline
$k$ & $N$ & \multicolumn{2}{c|}{first mode} & \multicolumn{2}{c|}{second mode} & \multicolumn{2}{c|}{fourth  mode} & \multicolumn{2}{c|}{eighth mode} \\[0.1cm] 
 & & error &  order & error &  order  & error &  order & error &  order  \\[0.1cm] \hline
 & 10 & 3.19e-2& ---	& 1.17e-1& ---	& 3.50e-1& ---	& 	6.99e-1&	 ---	 \\[0.1cm]
 & 20 & 8.16e-3&	1.97	& 3.19e-2&	1.87	& 1.17e-1&	1.58	& 3.50e-1&	1.00 \\[0.1cm]
0 & 40 & 2.05e-3&	1.99	& 8.16e-3&	1.97	& 3.19e-2&	1.87	& 1.17e-1&	1.58 \\[0.1cm]
 & 80 & 5.14e-4&	2.00	& 2.05e-3&	1.99	& 8.16e-3&	1.97	& 3.19e-2&	1.87 \\[0.1cm]
 & 160 & 1.28e-4&	2.00	& 5.14e-4&	2.00	& 2.05e-3&	1.99	& 8.16e-3&	1.97 \\[0.1cm] \hline
 & 10 &  1.10e-4& ---	& 1.81e-3& ---	& 3.25e-2& ---	& 4.01e-1& --- \\[0.1cm]
 & 20 & 6.78e-6&	4.01	& 1.10e-4&	4.05	& 1.81e-3&	4.16	& 3.25e-2&	3.63 \\[0.1cm]
1 & 40 & 4.23e-7&	4.00	& 6.78e-6&	4.01	& 1.10e-4&	4.05	& 1.81e-3&	4.16 \\[0.1cm]
 & 80 & 2.64e-8&	4.00	& 4.23e-7&	4.00	& 6.78e-6&	4.01	& 1.10e-4&	4.05 \\[0.1cm]
 & 160 & 1.65e-9&	4.00	& 2.64e-8&	4.00	& 4.23e-7&	4.00	& 6.78e-6&	4.01 \\[0.1cm] \hline
 & 10 & 1.15e-7& ---	& 	7.52e-6& ---	& 	5.28e-4& ---	& 6.08e-2& ---\\[0.1cm]
 & 20 & 1.79e-9&	6.01	& 1.15e-7&	6.03	& 7.52e-6&	6.13	& 5.28e-4&	6.85 \\[0.1cm]
2 & 40 & 2.78e-11&	6.01	& 1.79e-9&	6.01	& 1.15e-7&	6.03	& 7.52e-6&	6.13 \\[0.1cm]
 & 80 & 9.88e-14&	8.14	& 2.78e-11&	6.01	& 1.79e-9&	6.01	& 1.15e-7&	6.03 \\[0.1cm] \hline
\end{tabular}
\caption{Unit interval, relative eigenvalue errors, $\eta =1$.} 
\label{tab:hho1deve}
\end{table}

\subsection{Smooth eigenfunctions in 1D and 2D unit domains}
Let $\Omega =  (0,1)$ or $ \Omega = (0,1) \times (0,1)$ be the unit interval in 1D or the unit square in 2D, respectively. The 1D problem \eqref{eq:pdee} has exact eigenvalues $\lambda_{j} = j^2 \pi^2 $ and corresponding normalized eigenfunctions $u_{j}(x) = \sqrt{2} \sin(j \pi x)$ with $j=1,2,\cdots$, whereas the 2D problem \eqref{eq:pdee} has exact eigenvalues $\lambda_{jk} = \pi^2 (j^2 + k^2)$ and normalized eigenfunctions 
$
u_{jk}(x,y) = 2\sin(j \pi x) \sin(k \pi y)
$
with  $j, k = 1, 2, \cdots.$ 
We discretize the unit interval uniformly with $N\in\{10, 20,40,80,160\}$ elements and the unit square uniformly with $N \times N$ squares with $N\in\{4, 8, 16, 32, 64\}$. 
The default stabilization parameter of the HHO method is $\eta = 1$. 
The relative eigenvalue errors are reported in Table~\ref{tab:hho1deve} in 1D and in Table~\ref{tab:hho2deve2d} in 2D for the first, second, fourth, and eighth eigenvalues and for the polynomial degrees $k\in\{0,1,2\}$.  These tables show good agreement with the convergence order predicted by Theorem~\ref{thm:errest}, i.e., the convergence order for the eigenvalues is indeed $h^{2k+2}$. The $H^1$-seminorm errors on the first, second, fourth, and eighth eigenfunctions in 1D are reported in Table \ref{tab:hho1defe}. We observe a good agreement with the convergence order predicted by Corollary~\ref{cor:H1}, that is, the convergence order for the eigenfunctions in the $H^1$-seminorm is indeed $h^{k+1}$.

\begin{table}[ht]
\centering 
\begin{tabular}{| c | c || cc | cc | cc | cc |}
\hline
$k$ & $N$ & \multicolumn{2}{c|}{first mode} & \multicolumn{2}{c|}{second mode} & \multicolumn{2}{c|}{fourth  mode} & \multicolumn{2}{c|}{eighth mode} \\[0.1cm] 
 & & error &  order & error &  order  & error &  order & error &  order  \\[0.1cm] \hline
 & 4 & 2.51e-1&	---	& 5.11e-1& ---	& 6.36e-1& ---	& 7.39e-1&	 ---\\[0.1cm]
 & 8 & 7.70e-2&	1.70	& 2.16e-1&	1.24	& 2.51e-1&	1.34	& 4.08e-1&	0.86 \\[0.1cm]
0 & 16 & 2.04e-2&	1.92	& 6.57e-2&	1.72	& 7.70e-2&	1.70	& 1.33e-1&	1.61 \\[0.1cm]
 & 32 & 5.18e-3&	1.98	& 1.74e-2&	1.92	& 2.04e-2&	1.92	& 3.73e-2&	1.84 \\[0.1cm]
 & 64 & 1.30e-3&	1.99	& 4.41e-3&	1.98	& 5.18e-3&	1.98	& 9.62e-3&	1.96 \\[0.1cm] \hline
 & 4 & 2.27e-2&	---	& 1.62e-1&	---	& 3.32e-1&	---	& 5.10e-1& ---	 \\[0.1cm]
 & 8 & 1.45e-3&	3.97	& 9.75e-3&	4.06	& 2.27e-2&	3.87	& 6.35e-2&	3.01 \\[0.1cm]
1 & 16 & 9.15e-5&	3.98	& 5.96e-4&	4.03	& 1.45e-3&	3.97	& 3.90e-3&	4.02 \\[0.1cm]
 & 32 & 5.74e-6&	3.99	& 3.71e-5&	4.00	& 9.15e-5&	3.98	& 2.45e-4&	3.99 \\[0.1cm]
 & 64 & 3.59e-7&	4.00	& 2.32e-6&	4.00	& 5.74e-6&	3.99	& 1.54e-5&	4.00 \\[0.1cm]\hline
 & 4  & 5.71e-4& ---	& 8.46e-3& ---	& 4.91e-2& ---	& 2.31e-1& ---	 \\[0.1cm]
 & 8 & 8.63e-6&	6.05	& 1.07e-4&	6.30	& 5.71e-4&	6.43	& 2.33e-3&6.64 \\[0.1cm]
2 & 16 & 1.34e-7&	6.01	& 1.62e-6&	6.05	& 8.63e-6&	6.05	& 3.34e-5&6.12 \\[0.1cm]
 & 32 & 2.09e-9&	6.00	& 2.51e-8&	6.01	& 1.34e-7&	6.01	& 5.14e-7&6.02 \\[0.1cm]
 & 64 & 3.26e-11&	6.00	& 3.92e-10&	6.00	& 2.09e-9&	6.00	& 8.01e-9&6.00 \\[0.1cm]  \hline
\end{tabular}
\caption{Unit square, relative eigenvalue errors, $\eta =1$.}
\label{tab:hho2deve2d} 
\end{table}

\begin{table}[ht]
\centering 
\begin{tabular}{| c | c || cc | cc | cc | cc |}
\hline
$k$ & $N$ & \multicolumn{2}{c|}{first mode} & \multicolumn{2}{c|}{second mode} & \multicolumn{2}{c|}{fourth  mode} & \multicolumn{2}{c|}{eighth mode} \\[0.1cm] 
 & & error &  order & error &  order  & error &  order & error &  order  \\[0.1cm] \hline
 & 10 & 2.08e-1&	---	& 9.87e-1&	---	& 4.39e0&	---	& 1.61e+1&	--- \\[0.1cm]
 & 20 & 1.02e-1&	1.03	& 4.16e-1&	1.25	& 1.97e0&	1.15	& 8.78e0&	0.88 \\[0.1cm]
0 & 40 & 5.05e-2&	1.01	& 2.03e-1&	1.03	& 8.32e-1&	1.25	& 3.95e0&	1.15 \\[0.1cm]
 & 80 & 2.52e-2&	1.00	& 1.01e-1&	1.01	& 4.06e-1&	1.03	& 1.66e0&	1.25 \\[0.1cm]
 & 160 & 1.26e-2&	1.00	& 5.04e-2&	1.00	& 2.02e-1&	1.01	& 8.13e-1&	1.03 \\[0.1cm] \hline

 & 10 & 8.17e-3&	---	& 9.87e-2&	---	& 3.33e0&	---	& 1.54e+1&	--- \\[0.1cm]
 & 20 & 2.04e-3&	2.00	& 1.63e-2&	2.60	& 1.97e-1&	4.08	& 6.67e0&	1.21 \\[0.1cm]
1 & 40 & 5.11e-4&	2.00	& 4.09e-3&	2.00	& 3.27e-2&	2.60	& 3.95e-1&	4.08 \\[0.1cm]
 & 80 & 1.28e-4&	2.00	& 1.02e-3&	2.00	& 8.17e-3&	2.00	& 6.53e-2&	2.60 \\[0.1cm]
 & 160 & 3.19e-5&	2.00	& 2.55e-4&	2.00	& 2.04e-3&	2.00	& 1.63e-2&	2.00 \\[0.1cm] \hline

 & 10 & 2.17e-4&	---	& 4.36e-3&	---	& 1.52e0&	---	& 1.54e+1&	--- \\[0.1cm]
 & 20 & 2.71e-5&	3.00	& 4.34e-4&	3.33	& 8.71e-3&	7.45	& 3.05e0&	2.33 \\[0.1cm]
2 & 40 & 3.39e-6&	3.00	& 5.42e-5&	3.00	& 8.67e-4&	3.33	& 1.74e-2&	7.45 \\[0.1cm]
 & 80 & 4.24e-7&	3.00	& 6.78e-6&	3.00	& 1.08e-4&	3.00	& 1.73e-3&	3.33 \\[0.1cm]
 & 160 & 5.30e-8&	3.00	& 8.47e-7&	3.00	& 1.36e-5&	3.00	& 2.17e-4&	3.00 \\[0.1cm] \hline

 \end{tabular}
\caption{Unit interval, $H^1$-seminorm errors on eigenfunctions, $\eta =1$.}
\label{tab:hho1defe}
\end{table}

\subsubsection{Effect of the stabilization parameter $\eta$} \label{sec:supc}

We first report some striking superconvergence results for the HHO method with the stabilization parameter set to $\eta = 2k+3$ on 1D uniform meshes. In this case, we observe numerically two extra orders in the convergence of the relative eigenvalue errors, i.e., these errors now converge as $h^{2k+4}$, see Table \ref{tab:hho1devesup}. We thus obtain relative eigenvalue errors close to machine precision already on relatively coarse meshes. Moreover, we observe numerically (results are not reported for brevity) that taking values different from $2k+3$ for the stabilization parameter does not improve the relative eigenvalue errors. We also point out that the choice $\eta=2k+3$ does not increase the convergence order of the eigenfunctions. In 2D, we observe that the choice $\eta = 2k+3$ improves the approximation significantly in the sense of a much smaller constant $C$ in \eqref{eq:esterr_eve}, but the convergence order remains $h^{2k+2}$. The results are reported in Table \ref{tab:hho2deves} (compare with Table \ref{tab:hho2deve2d}). The theoretical analysis of the above observations is postponed to future work. 

\begin{table}[ht]
\centering 
\begin{tabular}{| c | c || cc | cc | cc | cc |}
\hline
$k$ & $N$ & \multicolumn{2}{c|}{first mode} & \multicolumn{2}{c|}{second mode} & \multicolumn{2}{c|}{fourth  mode} & \multicolumn{2}{c|}{eighth mode} \\[0.1cm] 
 & & error &  order & error &  order  & error &  order & error &  order \\[0.1cm] \hline
 & 10 &4.07e-5& --- & 6.59e-4& --- &1.10e-2& --- & 1.80e-1 & --- \\[0.1cm]
 & 20 &2.54e-6& 4.00	& 4.07e-5&	4.02	& 6.59e-4&	4.05	& 1.10e-2&	4.04 \\[0.1cm]
0 & 40 & 1.59e-7& 4.00	& 2.54e-6&	4.00	& 4.07e-5&	4.02	& 6.59e-4&	4.05 \\[0.1cm]
 & 80 & 9.91e-9& 4.00& 1.59e-7&4.00	& 2.54e-6&	4.00	& 4.07e-5&	4.02 \\[0.1cm]
 & 160 & 6.19e-10& 4.00& 9.91e-9&4.00 & 1.59e-7& 4.00 & 2.54e-6& 4.00 \\[0.1cm]\hline
 & 5 & 1.66e-6& --- & 1.13e-4& --- & 1.19e-2& --- & 1.74e-2& --- \\[0.1cm]
 & 10 & 2.55e-8& 6.02	& 1.66e-6&	6.09	& 1.13e-4&	6.72	& 1.19e-2&	0.54 \\[0.1cm]
1 & 20 & 3.98e-10& 6.00 & 2.55e-8& 6.02 & 1.66e-6& 6.09 & 1.13e-4& 6.72 \\[0.1cm]
 & 40 & 5.95e-12& 6.06 & 3.98e-10& 6.00 & 2.55e-8& 6.02 & 1.66e-6& 6.09 \\[0.1cm] \hline
 & 4 & 9.18e-9& --- &2.42e-6 & --- & 1.34e-2 & --- & 5.20e-1& ---	 \\[0.1cm]
2 & 8 & 3.57e-11&	 8.01 &9.18e-9& 8.04 & 2.42e-6& 12.43 &1.34e-2& 5.28 \\[0.1cm]
 & 16 & 1.04e-13& 8.42 & 3.57e-11& 8.01 & 9.18e-9 & 8.04 & 2.42e-6& 12.43 \\[0.1cm] \hline
\end{tabular}
\caption{Unit interval, relative eigenvalue errors, $\eta =2k+3$.}
\label{tab:hho1devesup}
\end{table}

\begin{table}[ht]
\centering 
\begin{tabular}{| c | c || cc | cc | cc | cc |}
\hline
$k$ & $N$ & \multicolumn{2}{c|}{first mode} & \multicolumn{2}{c|}{second mode} & \multicolumn{2}{c|}{fourth  mode} & \multicolumn{2}{c|}{eighth mode} \\[0.1cm] 
 & & error &  order & error &  order  & error &  order & error &  order  \\[0.1cm] \hline
 & 4 &  4.23e-2&	---	& 1.41e-1&	---	& 1.66e-1&	---	& 3.97e-1&	--- \\[0.1cm]
 & 8 & 1.06e-2&	1.99	& 3.60e-2&	1.97	& 4.23e-2&	1.97	& 7.84e-2&	2.34 \\[0.1cm]
0 & 16 & 2.66e-3&	2.00	& 9.04e-3&	1.99	& 1.06e-2&	1.99	& 1.98e-2&	1.99 \\[0.1cm]
 & 32 & 6.65e-4&	2.00	& 2.26e-3&	2.00	& 2.66e-3&	2.00	& 4.96e-3&	2.00 \\[0.1cm]
 & 64 & 1.66e-4&	2.00	& 5.66e-4&	2.00	& 6.65e-4&	2.00	& 1.24e-3&	2.00 \\[0.1cm]\hline
 & 4 &  2.74e-4&	---	& 3.33e-3&	---	& 5.80e-5&	---	& 1.73e-2&	--- \\[0.1cm]
 & 8 & 2.13e-5&	3.69	& 1.69e-4&	4.30	& 2.74e-4&	-2.24	& 1.75e-4&	6.63 \\[0.1cm]
1 & 16 & 1.40e-6&	3.93	& 9.93e-6&	4.09	& 2.13e-5&	3.69	& 3.47e-6&	5.66 \\[0.1cm]
 & 32 & 8.82e-8&	3.98	& 6.11e-7&	4.02	& 1.40e-6&	3.93	& 4.41e-7&	2.97 \\[0.1cm]
 & 64 & 5.53e-9&	4.00	& 3.80e-8&	4.01	& 8.82e-8&	3.98	& 3.11e-8&	3.83 \\[0.1cm]\hline
 & 4 &  1.75e-5&	---	& 3.33e-5&	---	& 8.23e-4&	---	& 1.28e-3&	--- \\[0.1cm]
 & 8 & 2.90e-7&	5.91	& 8.50e-7&	5.29	& 1.75e-5&	5.56	& 4.54e-5&	4.82 \\[0.1cm]
2 & 16 & 4.60e-9&	5.98	& 1.45e-8&	5.87	& 2.90e-7&	5.91	& 8.01e-7&	5.82 \\[0.1cm]
 & 32 & 7.20e-11&	6.00	& 2.32e-10&	5.97	& 4.60e-9&	5.98	& 1.29e-8&	5.96 \\[0.1cm]
 & 64 & 2.66e-13&	8.08	& 3.02e-12&	6.26	& 7.20e-11&	6.00	& 2.02e-10&	5.99 \\[0.1cm] \hline
\end{tabular}
\caption{Unit square, Relative eigenvalue errors, $\eta =2k+3$.}
\label{tab:hho2deves} 
\end{table}

In all of our numerical experiments, the default choice $\eta=1$ for the stabilization parameter produces satisfactory results. As expected, decreasing the value of $\eta$ progressively leads to a loss of stability in the HHO stiffness matrix, and therefore to a degradation of the accuracy of the discrete eigenvalues and eigenfunctions. To illustrate this simple fact, we report in Table \ref{tab:eta} the first four discrete eigenvalues using a polynomial degree $k\in\{0,1,2\}$ and a stabilization parameter $\eta = 2^{-l}$, $l\in\{0,\ldots,6\}$. We consider here a quasi-uniform sequence of triangular meshes with an initial average mesh-size $0.017$, where  the next finer mesh in the sequence is produced by dividing each triangle into four congruent sub-triangles. The results reported in Table \ref{tab:eta} indicate that the sensitivity to the choice of a too small value of $\eta$ swiftly decreases as the polynomial degree $k$ increases. A similar study varying the stabilization parameter in the context of the VEM can be found in \cite{mora2015virtual}, where the loss of accuracy also follows from the loss of stability if the value assigned to the stabilization parameter is too low.

\begin{table}[ht]
\centering 
\begin{tabular}{| c | c || c | c | c | c |}
\hline
$k$ & $\eta$ & $\lambda_{h,1}$ & $\lambda_{h,2}$ & $\lambda_{h,3}$ & $\lambda_{h,4}$  \\[0.1cm] \hline
 
&	1/4&	1.74e1&	3.70e1&	3.73e1&	5.11e1 \\[0.1cm] 
0&	1/2&	1.85e1&	4.25e1&	4.27e1&	6.26e1 \\[0.1cm] 
&	1&	1.92e1&	4.60e1&	4.60e1&	7.05e1 \\[0.1cm] \hline
&	1/64&	1.57e1&	5.09e1&	5.11e1&	8.16e1 \\[0.1cm]  
1&	1/8&	1.97e1&	4.87e1&	4.87e1&	7.54e1 \\[0.1cm] 
&	1&	1.97e1&	4.93e1&	4.93e1&	7.87e1 \\[0.1cm] \hline
&	1/64&	1.97e1&	4.94e1&	4.94e1&	7.90e1 \\[0.1cm] 
2&	1/8&	1.97e1&	4.93e1&	4.93e1&	7.89e1 \\[0.1cm] 
&	1&	1.97e1&	4.93e1&	4.93e1&	7.90e1 \\[0.1cm]  \hline

\end{tabular}
\caption{Discrete eigenvalues $\lambda_{h,j}$, $j\in\{1,2,3,4\}$ with polynomial degree $k\in\{0,1,2\}$ and stabilization parameter $\eta = 2^{-l}$, $l\in\{0,\ldots,6\}$.}
\label{tab:eta} 
\end{table}

 


\subsubsection{Polygonal (hexagonal) meshes in 2D}

To illustrate the fact that the same convergence orders can be obtained if the
HHO method is deployed on general meshes, we consider now a quasi-uniform sequence of polygonal (hexagonal) meshes of the unit square; see Figure \ref{fig:hex}.
The coarsest mesh in the sequence is composed of predominantly hexagonal cells with average mesh-size 0.065; the average mesh-size is halved from one mesh in the sequence to the next finer mesh. Table \ref{tab:hho2dhex} shows the relative eigenvalue errors for $k\in\{0,1,2\}$ with stabilization parameter $\eta=1$ and $\eta = 2k+3$ for the first ($j=1$) and third ($j=3$) eigenpairs. We observe a convergence of order $h^{2k+2}$, in agreement with Theorem~\ref{thm:errest}. Once again, the choice $\eta=2k+3$ for the stabilization parameter does not change the convergence order, but substantially improves the constant $C$. 

\begin{figure}[ht]
\centering
\includegraphics[width=6.6cm]{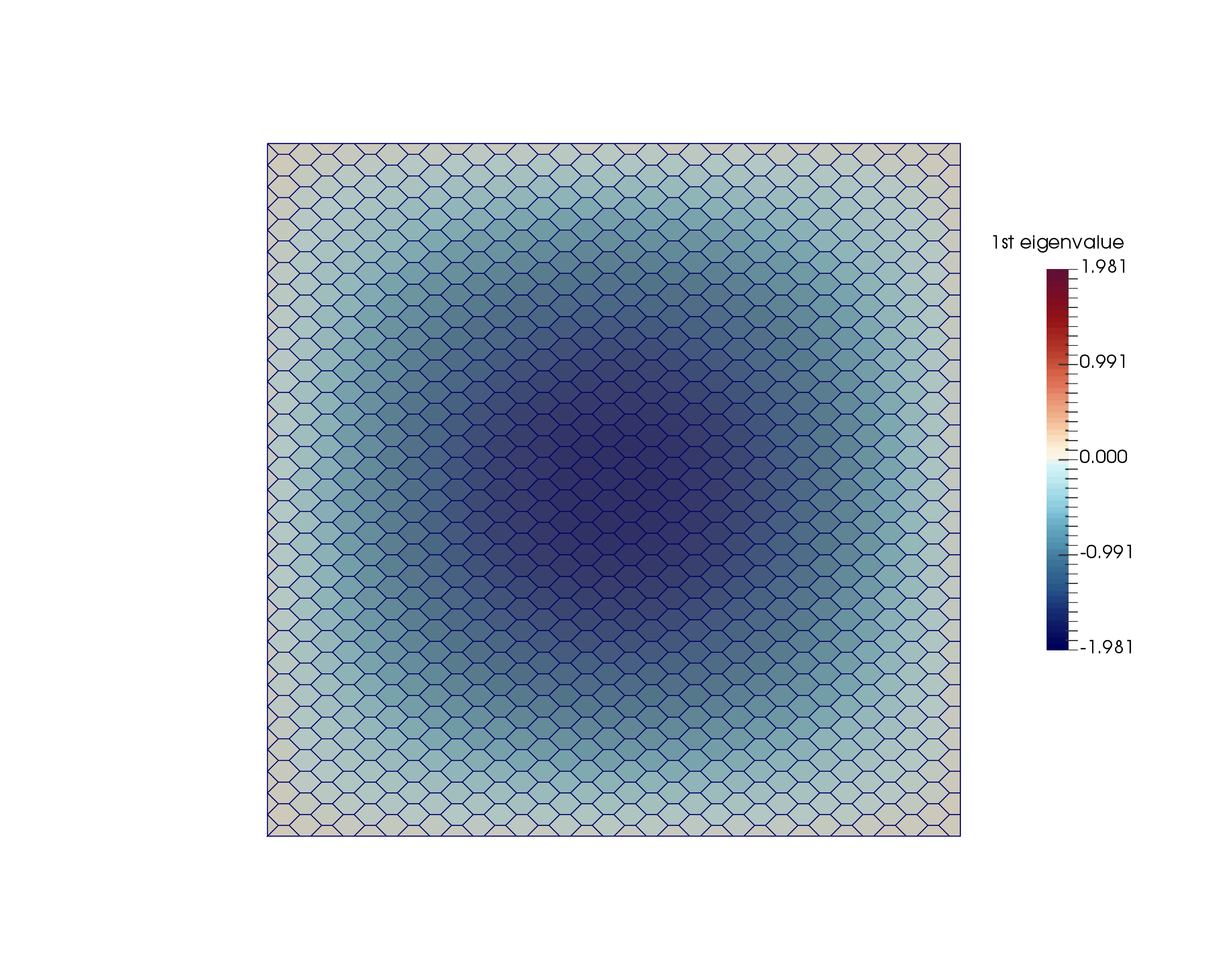} \hspace{0.1cm}
\includegraphics[width=6.6cm]{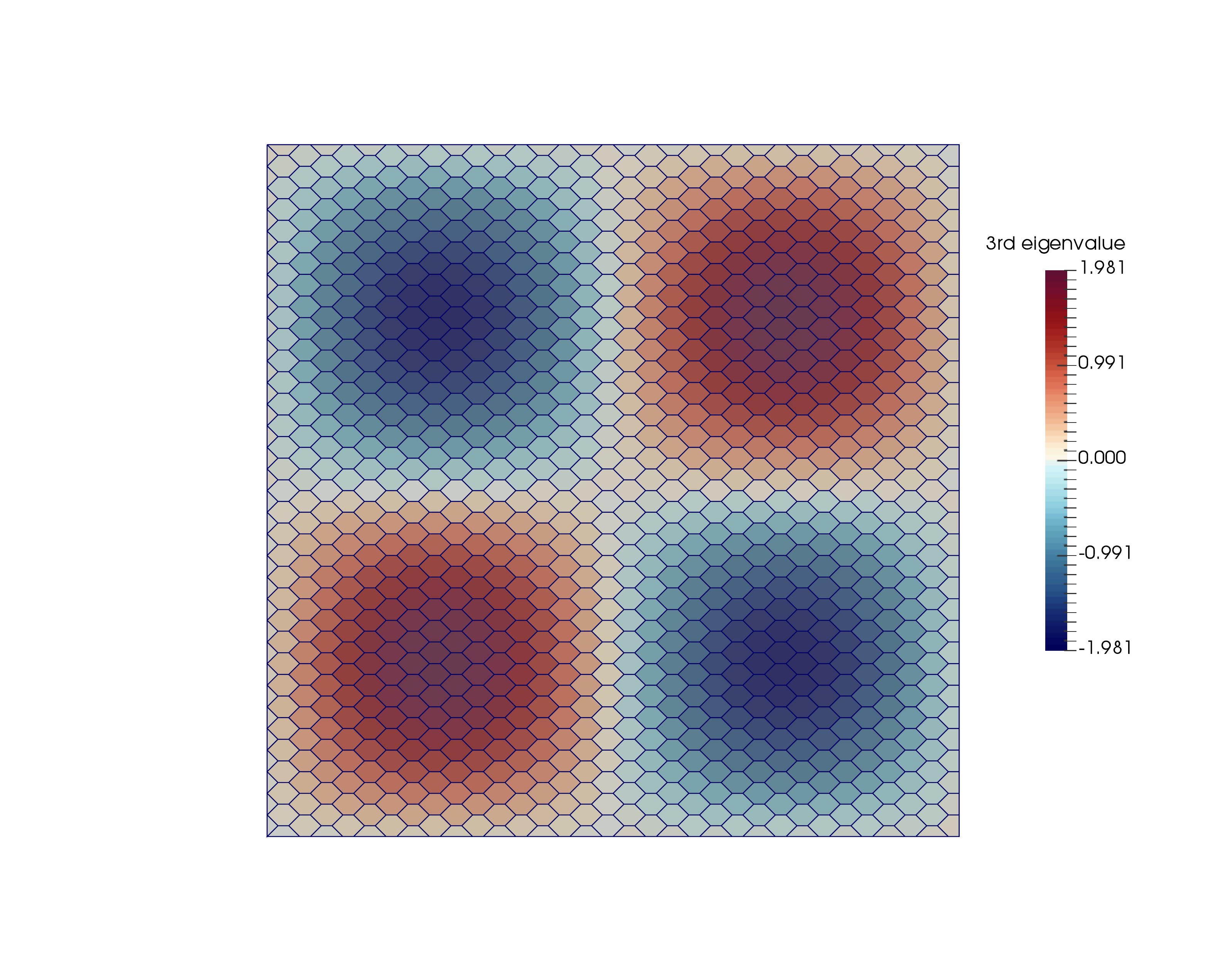} 
\caption{First and third approximate eigenfunctions with hexagonal meshes.}
\label{fig:hex}
\end{figure}

\begin{table}[ht]
\centering 
\begin{tabular}{| c | c || cc | cc || cc | cc |}
\hline
$k$ & $\ell$ & \multicolumn{2}{c|}{$j=1, \eta=1$} & \multicolumn{2}{c||}{$j=1, \eta=2k+3$} & \multicolumn{2}{c|}{$j=3, \eta=1$} & \multicolumn{2}{c|}{$j=3, \eta=2k+3$} \\[0.1cm] 
 & & error &  order & error &  order  & error &  order & error &  order  \\[0.1cm] \hline

 & 0 & 3.20e-1&	---	& 8.72e-2&	---	& 6.69e-1&	---	& 2.78e-1&	--- \\[0.1cm]
 & 1 & 1.19e-1&	1.43	& 2.21e-2&	1.98	& 3.52e-1&	0.93	& 8.41e-2&	1.73 \\[0.1cm]
0 & 2 & 3.56e-2&	1.74	& 5.45e-3&	2.02	& 1.28e-1&	1.45	& 2.15e-2&	1.97 \\[0.1cm]
 & 3 & 9.60e-3&	1.89	& 1.35e-3&	2.02	& 3.73e-2&	1.78	& 5.37e-3&	2.00 \\[0.1cm]
 & 4 & 2.48e-3&	1.95	& 3.34e-4&	2.01	& 9.84e-3&	1.92	& 1.33e-3&	2.01 \\[0.1cm] \hline

 & 0 & 1.97e-2&	---	& 1.10e-3&	---	& 3.16e-1&	---	& 1.43e-2&	--- \\[0.1cm]
 & 1 & 1.33e-3&	3.89	& 8.28e-5&	3.73	& 2.40e-2&	3.72	& 1.32e-3&	3.43 \\[0.1cm]
1 & 2 & 8.74e-5&	3.92	& 5.63e-6&	3.88	& 1.45e-3&	4.05	& 9.01e-5&	3.88 \\[0.1cm]
 & 3 & 5.64e-6&	3.95	& 3.66e-7&	3.94	& 9.11e-5&	3.99	& 5.86e-6&	3.94 \\[0.1cm]
 & 4 & 3.59e-7&	3.97	& 2.33e-8&	3.97	& 5.75e-6&	3.98	& 3.73e-7&	3.97 \\[0.1cm] \hline

 & 0 & 2.99e-4&	---	& 1.09e-5&	---	& 4.09e-2&	---	& 8.54e-4&	--- \\[0.1cm]
 & 1 & 5.15e-6&	5.86	& 2.26e-7&	5.59	& 3.63e-4&	6.82	& 1.42e-5&	5.91 \\[0.1cm]
2 & 2 & 8.55e-8&	5.91	& 3.97e-9&	5.83	& 5.59e-6&	6.02	& 2.53e-7&	5.81 \\[0.1cm]
 & 3 & 1.38e-9&	5.95	& 6.52e-11&	5.93	& 8.88e-8&	5.98	& 4.17e-9&	5.92 \\[0.1cm]
 & 4 & 2.21e-11&	5.97	& 9.32e-13&	6.13	& 1.41e-9&	5.98	& 6.69e-11&	5.96 \\[0.1cm] \hline
\end{tabular}
\caption{Unit square with hexagonal meshes, relative eigenvalue errors, $\eta =1 $ and $\eta =2k+3$.}
\label{tab:hho2dhex} 
\end{table}

\subsubsection{Comparison with the finite element method (FEM)} \label{sec:hhofem}

We now present a brief comparison between the discrete eigenvalues obtained using a continuous linear finite element method (FEM(1)) and the HHO method with $k=0$ and $k=1$ (referred to as HHO(0) and HHO(1), respectively). We consider the same quasi-uniform sequence of triangular meshes as in Section~\ref{sec:supc}, and we use the stabilization parameter $\eta=1$ or $\eta=8$ for HHO(0) and $\eta=1$ for HHO(1). Table \ref{tab:hho-fem} reports the errors for the first and eighth eigenvalues. All the reported convergence orders match the theoretical predictions. HHO(0) with $\eta=1$ leads to somewhat larger errors than FEM(1), but the situation is significantly reversed when using HHO(0) with $\eta=8$ or HHO(1) with $\eta=1$. We also mention that our numerical experiments show that the overall costs of FEM(1) and HHO(0) on various domains and mesh configurations are roughly the same.

\begin{table}[ht]
\centering 
\begin{tabular}{ | c | cccc | cccc |}
\hline
 $\ell$ & \multicolumn{4}{c|}{first eigenvalue} & \multicolumn{4}{c|}{eighth eigenvalue} \\[0.1cm]
 & \multicolumn{1}{c}{FEM(1)} & \multicolumn{2}{c}{HHO(0)} & \multicolumn{1}{c|}{HHO(1)}  & \multicolumn{1}{c}{FEM(1)} & \multicolumn{2}{c}{HHO(0)} & \multicolumn{1}{c|}{HHO(1)}   \\[0.1cm]   \hline
  & - - - &  $\eta=1$ & $\eta=8$ & $\eta=1$ &  - - - &  $\eta=1$ & $\eta=8$ & $\eta=1$  \\[0.1cm] \hline
 
0&1.75e-2&	2.92e-2&	9.72e-5&	1.48e-4&	1.15e-1&	1.64e-1&	2.76e-3&	7.28e-3 \\[0.1cm]
1&4.36e-3&	7.47e-3&	2.57e-5&	9.45e-6&	2.79e-2&	4.65e-2&	8.48e-4&	4.68e-4 \\[0.1cm]
2&1.09e-3&	1.88e-3&	6.81e-6&	5.97e-7&	6.94e-3&	1.20e-2&	2.21e-4&	2.99e-5 \\[0.1cm]
3&2.72e-4&	4.70e-4&	1.75e-6&	3.76e-8&	1.73e-3&	3.03e-3&	5.59e-5&	1.89e-6 \\[0.1cm]
4&6.81e-5&	1.18e-4&	4.40e-7&	2.36e-9&	4.33e-4&	7.59e-4&	1.40e-5&	1.19e-7 \\[0.1cm] \hline
order &2.00 &	1.99&	1.95&	3.99&	2.01&	1.95&	1.92&	3.98 \\[0.1cm]\hline
  
\end{tabular}
\caption{Comparison of eigenvalue errors for $\lambda_{h,1}$ and $\lambda_{h,8}$ when using FEM(1), HHO(0) with $\eta=1$ or $\eta=8$, and HHO(1) with $\eta=1$.}
\label{tab:hho-fem} 
\end{table}

\subsection{L-shaped domain}
We now study the Laplacian eigenvalue problem on the L-shaped domain $\Omega = \Omega_0 \backslash \Omega_1$, where $\Omega_0 = (0, 2) \times (0, 2)$ and $\Omega_1 = [1, 2] \times [1, 2]$. The L-shaped domain $\Omega$ has a reentrant corner at the point $(1,1)$, which results in possibly non-smooth eigenfunctions. In fact, the first eigenfunction is in $H^{1+t}(\Omega)$ with $t=\frac23-\epsilon$ with $\epsilon$ arbitrarily small, and the corresponding eigenvalue is $\lambda_1 = 9.6397238440219$ 
\cite{betcke2005reviving}. There are also smooth eigenfunctions. For example, the third eigenfunction is smooth and the corresponding eigenvalue is known exactly to be $\lambda_3 = 2\pi^2$. Figure \ref{fig:lshape} shows the HHO approximations (with $\eta = 1$) of the first and the third eigenfunctions using quasi-uniform triangulations of $\Omega$ and the polynomial degree $k=1$. 

\begin{table}[ht]
\centering 
\begin{tabular}{| c | c || cc | cc || cc | cc |}
\hline
$k$ & $N$ & \multicolumn{2}{c|}{$j=1, \eta=1$} & \multicolumn{2}{c|}{$j=1, \eta=2k+3$} & \multicolumn{2}{c|}{$j=3, \eta=1$} & \multicolumn{2}{c|}{$j=3, \eta=2k+3$} \\[0.1cm] 
 & & error &  order & error &  order  & error &  order & error &  order  \\[0.1cm] \hline
 & 4 & 2.36e-1&	---	& 1.25e-1&	---	& 3.60e-1&	---	& 1.82e-1&	--- \\[0.1cm]
 & 8 & 7.79e-2&	1.60	& 4.12e-2&	1.61	& 1.24e-1&	1.54	& 5.32e-2&	1.77 \\[0.1cm]
0 & 16 & 2.37e-2&	1.72	& 1.37e-2&	1.59	& 3.42e-2&	1.86	& 1.39e-2&	1.94 \\[0.1cm]
 & 32 & 7.32e-3&	1.70	& 4.75e-3&	1.53	& 8.77e-3&	1.96	& 3.52e-3&	1.98 \\[0.1cm]
 & 64 & 2.36e-3&	1.63	& 1.71e-3&	1.47	& 2.21e-3&	1.99	& 8.82e-4&	2.00 \\[0.1cm] \hline
 & 4 & 2.08e-2&	---	& 1.04e-2&	---	& 2.24e-2&	---	& 4.62e-3&	--- \\[0.1cm]
 & 8 & 5.92e-3&	1.81	& 4.12e-3&	1.34	& 1.37e-3&	4.04	& 2.77e-4&	4.06 \\[0.1cm]
1 & 16 & 2.18e-3&	1.44	& 1.64e-3&	1.33	& 8.50e-5&	4.01	& 1.72e-5&	4.01 \\[0.1cm]
 & 32 & 8.55e-4&	1.35	& 6.51e-4&	1.33	& 5.31e-6&	4.00	& 1.07e-6&	4.00 \\[0.1cm]
 & 64 & 3.39e-4&	1.34	& 2.58e-4&	1.33	& 3.32e-7&	4.00	& 6.71e-8&	4.00 \\[0.1cm] \hline
\end{tabular}
\caption{L-shaped domain, relative eigenvalue errors, $\eta =1 $ and $\eta =2k+3$.}
\label{tab:hho2dlshape} 
\end{table}

\begin{figure}[ht]
\centering
\includegraphics[width=6.2cm]{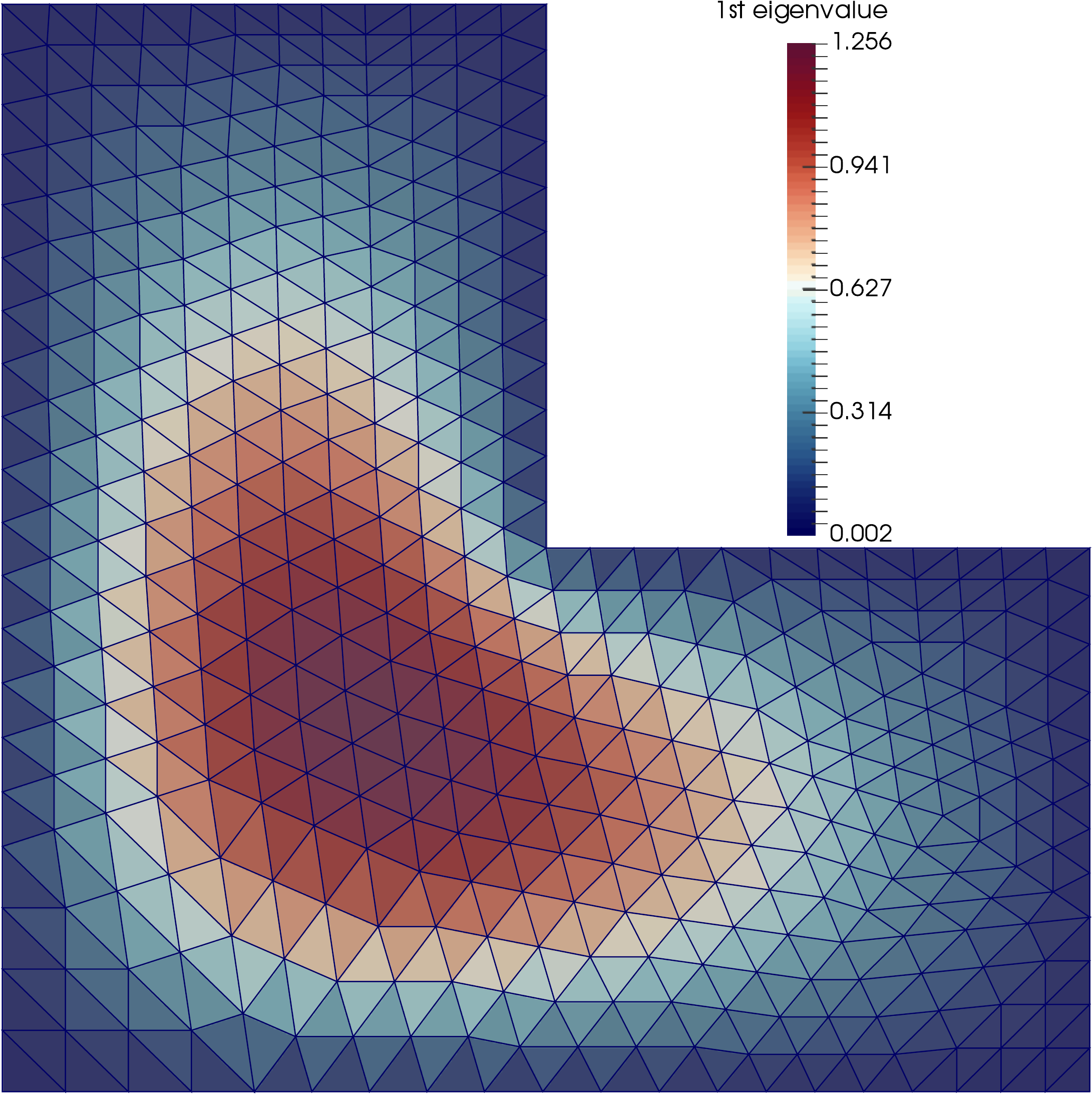} \hspace{0.2cm}
\includegraphics[width=6.2cm]{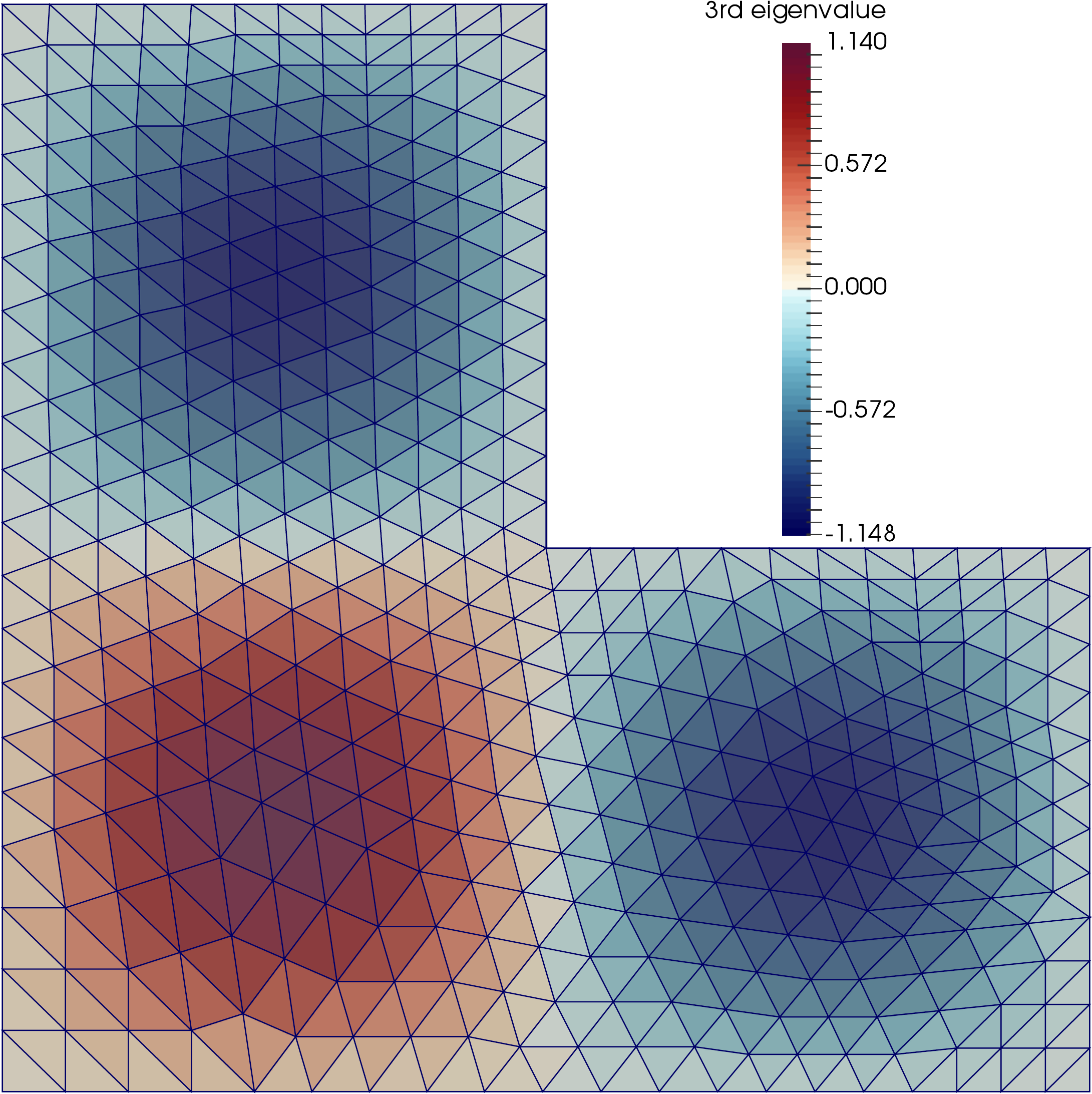} 
\caption{First and third approximate eigenfunctions in the L-shaped domain.}
\label{fig:lshape}
\end{figure}

To assess the convergence orders, we consider a sequence of triangulations where
each of the three unit squares composing the L-shaped domain $\Omega$ is discretized uniformly with $2\times N\times N$ triangular elements, where $N\in\{4,8,16,32,64\}$. Table \ref{tab:hho2dlshape} reports the relative eigenvalue errors for the first and third eigenvalues. We consider the values $\eta = 1$ and $\eta = 2k+3$ for the stabilization parameter together with the polynomial degrees $k\in\{0,1\}$. The relative error on the first eigenvalue converges with order $h^{2t}$ with $2t\approx\frac43$, whereas the relative error on the third eigenvalue converges with the optimal order $h^{2k+2}$. These results are again in agreement with Theorem~\ref{thm:errest}. The errors with $\eta = 2k+3$ are, as observed above, smaller than those with $\eta=1$. Comparing with the results reported in \cite{gopalakrishnan2015spectral} with HDG and $k=0$, the HHO approximation of the first eigenvalue converges with order $h^{4/3}$ whereas the HDG approximation converges with order $h$; the HHO approximation of the third eigenvalue converges with order $h^2$ whereas the HDG approximation converges with order $h$. 
Additionally, Table \ref{tab:lshapeH1} shows the eigenfunction errors in the $H^1$-seminorm for the first and third modes.  Here, the sequence of triangular meshes starts with an initial mesh-size 0.052, and the refinement procedure is the same as above. We use a linear FEM solution (normalized in $L^2(\Omega)$) solved at level 7 of the mesh sequence (this corresponds to a mesh-size $4.05 \times 10^{-4}$) as a reference solution for the calculation of the first eigenfunction error, and we use $u_3(x,y) = 2\sin(\pi x) \sin(\pi y)/\sqrt{3}$ (normalized in $L^2(\Omega)$) to compute the error on the third eigenfunction. In both cases, the error convergence rates are in good agreement with Theorem~\ref{thm:errest}.

\begin{table}[ht]
\centering 
\begin{tabular}{ | c | cccc | cccc |}
\hline
 $\ell$ & \multicolumn{4}{c|}{first mode} & \multicolumn{4}{c|}{third mode} \\[0.1cm] \hline
 & \multicolumn{2}{c}{$k=0$} & \multicolumn{2}{c|}{$k=1$} & \multicolumn{2}{c}{$k=0$} & \multicolumn{2}{c|}{$k=1$}  \\[0.1cm]   \hline
  & $\eta=1$ &  $\eta=3$ & $\eta=1$ &  $\eta=3$  & $\eta=1$ &  $\eta=3$ &  $\eta=1$ &  $\eta=3$ \\[0.1cm] \hline

0&	1.38e0&	1.33e0&	3.80e-1&	3.30e-1&	2.36e0&	2.16e0&	4.73e-1&	4.36e-1 \\[0.1cm]
1&	7.42e-1&	7.31e-1&	2.21e-1&	1.98e-1&	1.20e0&	1.14e0&	1.09e-1&	1.07e-1 \\[0.1cm]
2&	4.09e-1&	4.08e-1&	1.40e-1&	1.26e-1&	5.88e-1&	5.79e-1&	2.68e-2&	2.66e-2 \\[0.1cm]
3&	2.34e-1&	2.34e-1&	9.09e-2&	8.28e-2&	2.92e-1&	2.90e-1&	6.65e-3&	6.63e-3 \\[0.1cm]
4&	1.39e-1&	1.39e-1&	6.18e-2&	5.68e-2&	1.46e-1&	1.45e-1&	1.62e-3&	1.62e-3 \\[0.1cm] \hline
order &	0.83&	0.81&	0.65&	0.63&	1.01&	0.98&	2.04&	2.01 \\[0.1cm] \hline
  
\end{tabular}
\caption{L-shaped domain, first and third eigenfunction errors in the $H^1$-seminorm with polynomial degree $k\in\{0,1\}$ and stabilization parameter $\eta\in\{1, 2k+3\}$. }
\label{tab:lshapeH1} 
\end{table}

\subsection{Unit disk}
Lastly, we consider the Laplacian eigenvalue problem \eqref{eq:pdee} in the unit disk $\Omega = \{ (x,y): x^2 + y^2 \le 1\}$. 
Using polar coordinates, the eigenpairs are
\begin{equation}
((s_{n,m}^2, J_n( s_{n,m} r ) \cos(n \theta))_{n=0,1,2,\cdots}, \qquad
((s_{n,m}^2, J_n( s_{n,m} r ) \sin(n \theta))_{n=1,2,\cdots},
\end{equation}
where $J_n$ is the Bessel function of order $n$, and $s_{n,m}$ are the zeros of the Bessel functions with $m=1,2,3,\cdots$.
Figure \ref{fig:circle} shows the first and seventh discrete eigenfunctions.

\begin{figure}[ht]
\centering
\includegraphics[width=6.2cm]{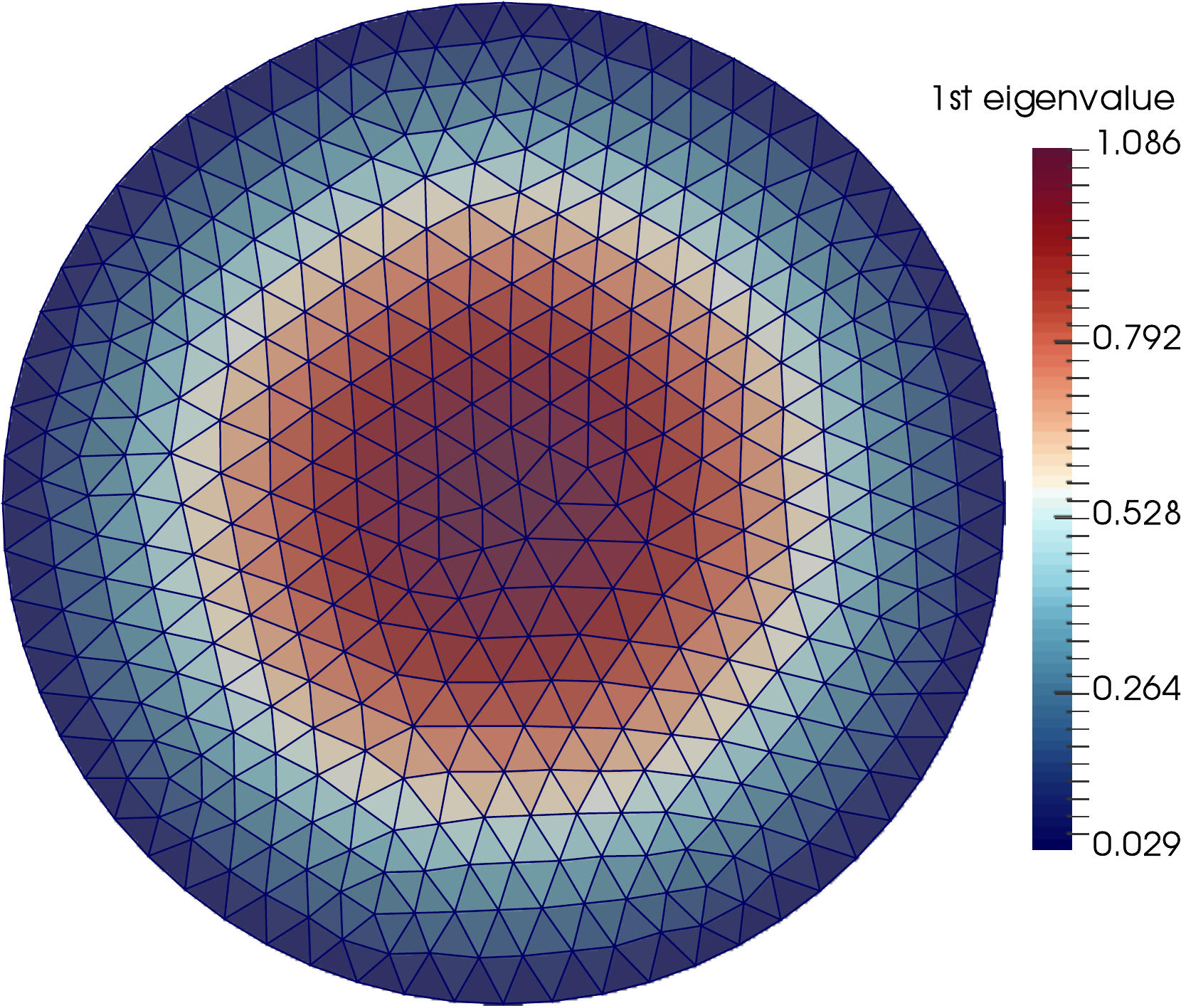} \hspace{0.2cm}
\includegraphics[width=6.2cm]{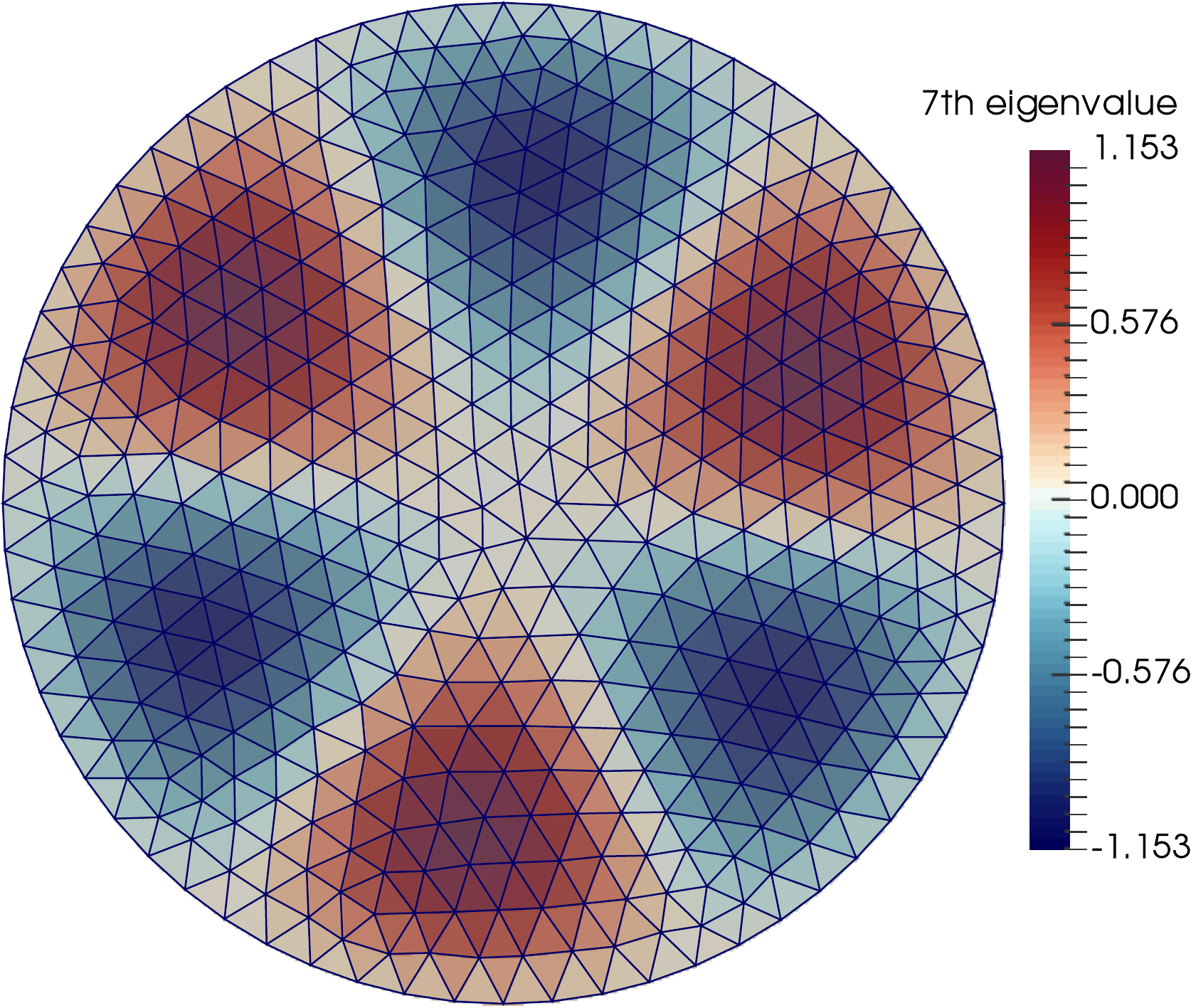} \\
\caption{The first and seventh approximate eigenfunctions in the unit disk. }
\label{fig:circle}
\end{figure}

\begin{table}[ht]
\centering 
\begin{tabular}{| c | c || cc | cc | cc |}
\hline
$k$ & $\ell$ & \multicolumn{2}{c|}{first mode} & \multicolumn{2}{c|}{fourth mode} & \multicolumn{2}{c|}{seventh  mode} \\[0.1cm] 
 & & error &  order & error &  order  & error &  order  \\[0.1cm] \hline
 & 0 & 6.35e-3& 	---	& 1.55e-2& 	---	& 3.19e-2& 	--- \\[0.1cm]
 & 1 & 1.59e-3& 	2.00	& 3.82e-3& 	2.02	& 7.95e-3& 	2.00 \\[0.1cm]
 & 2 & 3.98e-4& 	2.00	& 9.54e-4& 	2.00	& 1.99e-3& 	2.00 \\[0.1cm]
0 & 3 & 9.96e-5& 	2.00	& 2.38e-4& 	2.00	& 4.97e-4& 	2.00 \\[0.1cm]
 & 4 & 2.49e-5& 	2.00	& 5.96e-5& 	2.00	& 1.24e-4& 	2.00 \\[0.1cm]
 & 5 & 6.22e-6& 	2.00	& 1.49e-5& 	2.00	& 3.11e-5& 	2.00 \\[0.1cm] \hline
\end{tabular}
\caption{Unit disk, relative eigenvalue errors, $k=0$ and $\eta =3$.}
\label{tab:hho2dcirc} 
\end{table}

We approximate the unit disk using a sequence of unstructured triangulations
where the coarsest mesh in the sequence ($\ell=0$) is composed of triangular cells with mesh-size 0.033, and the refinement procedure is the same as above.
Since the boundary of the disk is approximated by straight lines, the error committed by this discretization is of order $h^2$. Thus, we only consider the lowest-order HHO approximation with $k=0$.
Table \ref{tab:hho2dcirc} reports the relative eigenvalue errors with $\eta = 2k+3=3$ for the stabilization parameter. We observe a convergence order of $h^2$ as predicted.


\else
\fi

\section{Concluding remarks} \label{sec:conclusion}
In this paper, we devised and analyzed the approximation of the eigenvalues and eigenfunctions of a second-order selfadjoint elliptic operator using the Hybrid High-Order (HHO) method. Using polynomials of degree $k\ge0$ for the face 
unknowns, and assuming smooth eigenfunctions, we established theoretically and observed numerically that the errors on the eigenvalues converge as $h^{2k+2}$ whereas the errors on the eigenfunctions converge as $h^{k+1}$ in the $H^1$-seminorm. We considered triangular and polygonal (hexagonal) meshes in the numerical experiments for the Laplace eigenproblem in two-dimensional domains with smooth and non-smooth eigenfunctions. Additionally, we observed numerically in one dimension that the eigenvalue error converges at the even faster rate $h^{2k+4}$ for the particular choice $\eta=2k+3$ of the stabilization parameter in the HHO method. Several extensions of the present work can be considered, among which we mention biharmonic eigenvalue problems and non-selfadjoint second-order eigenvalue problems as well as the Maxwell eigenvalue problem in a curl-curl setting.

\section*{Acknowledgements}
This publication was made possible in part by the CSIRO Professorial Chair in Computational Geoscience at Curtin University and the Deep Earth Imaging Enterprise Future Science Platforms of the Commonwealth Scientific Industrial Research Organisation, CSIRO, of Australia. Additional support was provided by the European Union's Horizon 2020 Research and Innovation Program of the Marie Sk{\l}odowska-Curie grant agreement No. 777778 and the Curtin Institute for Computation.
This work was initiated when the first author was visiting
the Institute Henri Poincar\'e (IHP) during the Fall 2016 Thematic
Trimester on ``Numerical Methods for Partial Differential Equations''. 
The support of the Institute is gratefully acknowledged. Stimulating discussions with E. Burman (University College London) and J.-L. Guermond (Texas A\&M University) who were also visiting the Institute for the Fall 2016 Thematic Quarter
are gratefully acknowledged.

\bibliographystyle{siam}
\bibliography{ref}

\end{document}